\documentclass[12pt]{amsart}
\usepackage[utf8]{inputenc}
\usepackage[T1]{fontenc}
\usepackage{lmodern}
\usepackage[english]{babel}
\usepackage{amsmath,mathtools,amsthm,amssymb,amsfonts}
\usepackage{color}
\usepackage{hyperref}
\usepackage{mathrsfs}
\usepackage{graphicx}
\usepackage{enumitem}
\usepackage{nicefrac}
\usepackage{cleveref}
\usepackage{accents}

\newcommand{\definedas}{\mathrel{\raise.095ex\hbox{\rm :}\mkern-5.2mu=}}

\newcommand{\R}{\mathbb{R}}

\newcommand{\Sbb}{\mathbb{S}}

\renewcommand{\d}{\,\mathrm{d}}

\newcommand{\btr}[1]{\left\vert#1\right\vert}

\newcommand{\Ric}{\mathrm{Ric}}

\newcommand{\tr}{\text{tr}}

\newcommand{\pr}{\partial}
\newcommand{\veps}{\varepsilon}

\newcommand{\lm}[2]{\lim\limits_{#1\to #2}}

\def\tr{\textmd{tr}}
\def\dv{\textnormal{div}}

\def\Lap{\Delta}

\def\Ric{\textnormal{Ric}}

\def\dint{\displaystyle\int}
\def\R{\mathbb{R}}

\def\R{\mathbb{R}}

\def\la{\langle}
\def\ra{\rangle}
\def\S{\Sigma}
\def\({\left(}
\def\){\right)}

\def\={\stackrel{*}{=}}
\def\To{\longrightarrow}

\def\s{\sigma}

\def\w{\omega}
\def\g{\gamma}

\def\ADM{\textnormal{ADM}}
\def\bS{\mathbb{S}}

\def\defeq{\coloneqq}

\def\bmu{\boldsymbol\mu}
\def\bJ{\mathbf{J}}

\def\wt{\widetilde{\tau}}
\def\AF{\textnormal{AF}}
\def\bc{\mathbf{c}}

\newcommand{\ddt}[2]{\frac{d #1}{d #2}}
\newcommand{\dds}[2]{\frac{\pr #1}{\pr #2}}

\theoremstyle{Definition}
\newtheorem{thm}{Theorem}[section]
\newtheorem{prop}[thm]{Proposition}

\theoremstyle{definition}
\newtheorem{defi}[thm]{Definition}
\newtheorem{lem}[thm]{Lemma}

\newtheorem{bem}[thm]{Remark}
\newtheorem{kor}[thm]{Corollary}

\newtheorem{conj}[thm]{Conjecture}

\begin{document}
\title[Families of non time-symmetric IDS and Penrose-like inequalities]{Families of non time-symmetric initial data sets and Penrose-like energy inequalities}
\begin{abstract}
Motivated by solving the constraint equations in the evolutionary form suggested by R\'acz in \cite{Racz}, we propose a family of asymptotically flat initial data sets which are ``asymptotically spherically symmetric'' at infinity. Within this family, we obtain Penrose-like energy estimates and establish the existence of solutions for the constraint equations in the spherical symmetric and totally umbilic cases.
\end{abstract}
\author[{Cabrera~Pacheco}]{ Armando J. {Cabrera~Pacheco}}
\address{Department of Mathematics, Universit\"at T\"ubingen, 72076 T\"ubingen, Germany}
\curraddr{T\"ubingen AI Center, Universit\"at T\"ubingen, 72076 T\"ubingen, Germany}
\email{a.cabrera@uni-tuebingen.de}

\author[Wolff]{Markus Wolff}
\address{Department of Mathematics, Universit\"at T\"ubingen, 72076 T\"ubingen, Germany}
\email{wolff@math.uni-tuebingen.de}
\maketitle
\thispagestyle{empty}	

	\section{Introduction}
	
	The fundamental work of Choquet-Bruhat \cite{Choquet} provided the foundations for the study of \emph{initial data sets} for the Einstein equations, that is, spacelike slices of spacetimes $(\mathbf{M},\mathbf{g})$. In this context, the Einstein equations
	\begin{equation}\label{eq-EE}
	\Ric(\mathbf{g}) - \frac{1}{2} R(\mathbf{g})\mathbf{g} + \Lambda \mathbf{g} = 8 \pi \mathbf{T}
	\end{equation}
	are seen as an initial value problem, where $\Ric(\mathbf{g})$ denotes the Ricci tensor and $R(\mathbf{g})$ the scalar curvature of $\mathbf{g}$, $\Lambda$ the cosmological constant and $\mathbf{T}$ the energy-momentum tensor. The idea is to solve \eqref{eq-EE} by providing appropriate initial values, consisting of a Riemannian manifold $(M,g)$ and symmetric (0,2)-tensor $K$ satisfying the so called \emph{constraint equations} (see Equations \eqref{eq-CC-ham} and \eqref{eq-CC-mom} below). The constraint equations constitute a PDE system themselves and their solutions $(M,g,K)$ are referred to as \emph{initial data sets}. The results (for $\mathbf{T}=0$) of Choquet-Bruhat ensure the existence of a maximal solution, and hence a (vacuum) spacetime $(\mathbf{M},\mathbf{g})$ containing $(M,g)$ as a spacelike slice with second fundamental form $K$. Subsequent research has been carried out to extend this result for the $\mathbf{T} \neq 0$ case (see \cite{CB-book}). Having the existence of spacetimes as the evolution of an initial data set at hand, a natural question arises: can we find physically meaningful solutions to the constraint equations? We refer to the reader to a recent survey by Carlotto \cite{Carlotto} for an comprehensive description of the state of the art of this fundamental problem in mathematical relativity.
	
	In this work we focus on isolated systems, that is, initial data sets $(M,g,K)$ for which the matter content is contained in a finite region and which have specific decay conditions at infinity (leading to either \emph{asymptotically flat}, \emph{asymptotically hyperbolic} or \emph{asymptotically hyperboloidal} initial data sets, see Section~\ref{sec_prelim} for the corresponding definitions) making appropriate notions of total mass and total momentum well defined. Moreover, we assume that for $n \geq 3$, $M = [r_0, \infty) \times \bS^{n-1}$ and that
		\begin{align}
		g &= N(r,\cdot)^2 dr^2 + r^2 \s(r,\cdot),  \label{eq-gintro}\\
		K &= k(r,\cdot) N(r,\cdot)^2 dr^2 + p(r,\cdot) r^2 \s(r,\cdot). \label{eq-Kintro}
		\end{align}
		where  $N \defeq N(r,\cdot)$, $k \defeq k(r,\cdot)$ and $p \defeq p(r,\cdot)$ are smooth functions on $M$ ($r_0 \leq r < \infty$), and $\{ \s(r,\cdot) \}_{r_0 \leq r < \infty}$ is a smooth family of metrics on $\bS^{n-1}$.
		
		The time-symmetric case $(K=0)$ for $n=3$ was considered by Bartnik in \cite{Bartnik-QS}. By considering $\s(r,\cdot)=\s_{\bS^2}$, where $\s_{\bS^2}$ denotes the standard round metric on $\bS^2$, he showed that the \emph{lapse} $N$ satisfies a parabolic equation (actually he did it for a more general form of the metric) and gave sufficient conditions for a solution to exists, and for it to give rise to an asymptotically flat (time-symmetric) initial data set. In fact, his construction yields a family of time-symmetric initial data sets for which the Riemannian Penrose inequality holds, which was not yet proven at the time. Further extensions and adaptations of Bartnik's method include the works of Lin \cite{Lin}, Jang \cite{Jang}, Shi--Tam \cite{S-T}, Smith \cite{Smith} and Smith--Weinstein \cite{SW1,SW2}. For the non time-symmetric case $(K \neq 0)$ not much is known, either about solutions to the constraint equations (see \cite{Carlotto}) or the Penrose inequality (see \cite{Mars}). For the general case $K \neq 0$, R\'acz in \cite{Racz} recasts the constraint equations as an evolutionary problem. R\'acz' setting reduces to Bartnik's equation in the time-symmetric case. Moreover, R\'acz shows short time existence of solutions, providing the first steps towards solving the constraint equations in a general form in this context, however, the long time existence and conditions to guarantee solutions to be asymptotically flat were posed as an open problem. A different approach to this problem was studied by Sharples \cite{Sharples}.
		
		Inspired by the works of Bartnik \cite{Bartnik-QS} and R\'acz \cite{Racz}, we construct new families of (non time-symmetric) initial data sets arising as solutions of the constraint equations. More precisely, by setting the family of metrics $\{ \s(r,\cdot) \}_{r_0 \leq r < \infty}$ to be given by a geometric flow on $\bS^{n-1}$ converging exponentially fast to the round metric $\s_{\bS^{n-1}}$, we show that in the totally umbilic case ($p=k$ in \eqref{eq-Kintro}) the constraint equations are solvable by prescribing appropriate components of the energy-momentum tensor $\mathbf{T}$, namely, the \emph{energy density} and the radial component of the \emph{momentum density} (see Theorem~\ref{thm-existence-totally-umbilical}). Similarly, for the spherically symmetric case, the constraint equations decouple and become a system of ODE's with global solutions exhibiting the desired decay (see Proposition \ref{prop_sphericalsymmetry}). Although existence of solutions to the constraint equations in spherically symmetry are known to exists and have been widely studied (see \cite{Mars} and references therein), we believe that it is important to study them in the context of the system of equations obtained by R\'acz in \cite{Racz}. We mainly perform our analysis for asymptotically flat manifolds, and for $n=3$, but we also provide the tools needed for the general case $n \geq 3$ and the asymptotically hyperbolic case ($\Lambda < 0$).
		
		In a similar way to \cite{Bartnik-QS}, by studying manifolds with this particular structure we are able to construct families of initial data sets for which new  Penrose-like energy inequalities hold. Here, by this we mean an inequality of the form
		\begin{equation*}
		E_{\ADM} \geq \sqrt{\frac{|\pr M|}{16 \pi}}
		\end{equation*}
		where $E_{\ADM}$ is the ADM energy. In spherical symmetry the ADM mass, $m_{\ADM}$, coincides with $E_{\ADM}$ since the momentum tensor vanishes (see \cite{Lee-GRbook} for example), thus in this case the energy inequality actually implies the Riemannian Penrose inequality. From this point of view, our results can be considered as extending the Penrose (energy) inequality known in spherical symmetry to cases where spherical symmetry does not hold. However, the manifolds considered here are ``rotationally symmetric near infinity'', in the sense that the metric on $\bS^2$ is very close to the round metric at infinity. More precisely, we prove the following.

		\begin{thm}\label{thm-main}
			Let $(M,g,K)$ be an asymptotically flat initial data set of the form \eqref{eq-gintro}-\eqref{eq-Kintro} and $\{ \S_r \}_{r \in [r_0,\infty)} \subset M$, where $\S_r \defeq \{r\} \times \bS^2$. Suppose that $\{\s(r)\}_{r\in[r_0,\infty)}$ is a smooth family of metrics generated by a volume-preserving geometric flow that converges exponentially to the standard metric on $\bS^2$. Suppose that the inner boundary, $\S_{r_0}$, is a generalized apparent horizon and that the family $\{ \S_r \}_{r \in (r_0,\infty)}$ satisfies a strictly untrapped condition \emph{(}see Equation~\eqref{eq_strictlyuntrapped}\emph{)}. If the dominant energy condition holds, then the Hawking energy of $\S_r$, $m_H(\S_r)$, is non-decreasing and 		
			\begin{equation} \label{eq-PTineq}
			E_{\ADM} \geq  \sqrt{\frac{|\S_{r_0}|}{16 \pi}}.
			\end{equation}
			Additionally, there exists a non-negative function $f$ defined on $M$ such that the following refined energy Penrose inequality holds:
			\begin{equation} \label{eq-refinedPTineq}
			E_{\ADM} \ge  \sqrt{\frac{|\S_{r_0}|}{16 \pi}}+\int_{r_0}^{\infty} \int_{\bS^2} r^2 f(r,\cdot) \, dA_{\s(r)}\,dr,
			\end{equation}
			with equality if and only if $(M,g,K)$ is isometric to a spacelike, spherically symmetric slice of a Schwarzschild spacetime.
		\end{thm}
		
		The function $f$ in Theorem~\ref{thm-main} can be considered as a quantity measuring the angular components of the energy-momentum tensor $\mathbf{T}$. We remark that the condition on $\{\s(r)\}_{r\in[r_0,\infty)}$ is natural. In fact, in the work of Bartnik \cite{Bartnik-QS} $,\s(r)=r^2 \s_{\bS^2}$, where $\bS^2$ denotes the standard round metric on $\bS^2$, while in \cite{Lin} and \cite{Jang}, Hamilton's modified Ricci flow is used.
		\\

	This work is organized as follows. In Section~\ref{sec_prelim}, we introduced the family of initial data sets $\mathcal{M}$ and obtain the constraint equations in the evolutionary form suggested by R\'acz \cite{Racz}. In Section~\ref{sec_penrose} we establish Penrose-like energy estimates by making use of the monotonicity of the Hawking energy. In Section~\ref{sec_solving} we study the existence of solutions to the constraint equations within the family $\mathcal{M}$ in the asymptotically flat, hyperbolic and hyperboloidal cases. Finally in Section~\ref{sec_discussion} we discuss potential future works.
	
			\subsection*{Acknowledgements.}
			
			The work of AJCP was funded by Carl Zeiss Foundation and the financial support of the Deutsche Forschungsgemeinschaft through the SPP 2026 “Geometry at Infinity”. Aspects of this work were completed as part of the STSM ``Analytic study of asymptotically flat systems in general relativity with parabolic constructions'' in the context of the European COST Action CA16104 on Gravitational waves, black holes and fundamental physics (GWverse). The authors would like to thank Florian Beyer, Carla Cederbaum, Greg Galloway, Jim Isenberg, Marc Mars and Todd Oliynyk for useful and interesting discussions related to this work. AJCP wants to especially thank Istv\'an R\'acz for his encouragement, support  and hospitality at different stages of this work.
	
	\section{Preliminaries}\label{sec_prelim}
		An \emph{initial data set} is a triple $(M,g,K)$ given by an $n$-dimensional Riemannian manifold $(M,g)$ ($n\ge 3$) and a symmetric $(0,2)$-tensor $K$ that satisfy the \emph{constraint equations}:
		\begin{align} 
		R(g) + (\tr_{g} K)^2 - |K|^2_{g} - 2\Lambda &= 16 \pi \bmu,  \label{eq-CC-ham} \\
		\dv_{g}(K - ((\tr_{g} K) g) )& = 8 \pi \bJ, \label{eq-CC-mom}
		\end{align}
		for a smooth function $\bmu$ and a one-form $\bJ$ on $M$, where $R(g)$ denotes the scalar curvature of $(M,g)$ and $\Lambda$ is a constant known as the \emph{cosmological constant}. We call $\bmu$ the \emph{energy density} and $\bJ$ the \emph{momentum density}, which are usually determined by the matter model of interest. In fact, if $M$ arises as a spacelike hypersurface with induced metric $g$ and second fundamental form $K$ with respect to a timelike, future-pointing unit normal $\vec{n}$ of an ambient spacetime $(\mathbf{M},\mathbf{g})$ satisfying the Einstein Equations (with cosmological constant $\Lambda$) with energy-stress tensor $\mathbf{T}$, then $(M,g,K)$ satisfies the constraint equations with energy density $\bmu:=\mathbf{T}(\vec{n},\vec{n})$ and momentum density $\bJ:=\mathbf{T}(\vec{n},\cdot\vert_{TM})$. Further, we say that $(M,g,K)$ satisfies the \emph{dominant energy condition} (DEC), if
		\begin{align}\label{eq_DEC}
			\bmu\ge \btr{\bJ}_g.
		\end{align}
		\begin{defi}\label{defi_asymflat}
			Let $(M,g,K)$ be an initial data set. We say that $(M,g,K)$ is \emph{asymptotically flat}, if $\Lambda=0$, $\bmu$ and $\bJ$ are integrable on $M$, and there exists a compact set $\Omega\subset M$ such that $M\setminus \Omega$ is diffeomorphic to $\R^n \setminus B_R(0)$ for some $R>0$, and with respect to these coordinates $g$ and $K$ satisfy the following decay conditions:
			\begin{align*}
			g_{ij} &=\delta_{ij} + O_2(|x|^{-a}), \\
			K_{ij} &=O_2(|x|^{-b}),
			\end{align*}
			for $a>\frac{n-2}{2}$ and $b>\frac{n}{2}$.
		\end{defi}
		Under these assumptions, there are well-defined notions of total energy and momentum~\cite{ADM}, called the \emph{ADM energy} and \emph{ADM momentum} defined as
		\begin{align*}
		E_{\ADM} &\defeq \frac{1}{2(n-1)\w_{n-1}} \lm{\rho}{\infty} \int_{\Sbb_{\rho}} (g_{ij,i} - g_{ii,j})\nu^j \, dV_{S_{\rho}}, \\
		\textbf{P}_i &\defeq -\frac{1}{(n-1)\w_{n-1}} \lm{\rho}{\infty}  \int_{\Sbb_{\rho}} (K_{ij} -  (\tr_g K)g_{ij})\nu^j \, dV_{S_{\rho}},
		\end{align*}
		for $i=1,...,n$, where the integrals are taken over the round spheres $\Sbb_\rho$ in $\R^{n}$ centered at the origin and $\nu$ is the standard unit normal to $\Sbb_\rho$. If $E=E_{\ADM}\ge |\textbf{P}|$ (where $|\,\cdot\,|$ denotes the Euclidean norm), we further define the \emph{ADM mass} as
		\[
			m=\sqrt{E^2-|\textbf{P}|^2}.
		\]
		Let $\Sigma$ be an orientable hypersurface in $(M,g,K)$.  Let $H$ denote the mean curvature of $\Sigma$ with respect to a unit normal $\nu$.
		We then define the null expansions $\theta_\pm$ of $\Sigma$ in $(M,g,K)$ as
		\begin{align*}
		\theta_\pm:=H\pm P,
		\end{align*}
		where $P \defeq \tr_{\Sigma}K=\tr_MK-K(\nu,\nu)$. If $\theta_\pm=0$ along $\Sigma$, we call $\Sigma$ a \emph{marginally outer / inner trapped surface} (MOTS/MITS) in $(M,g,K)$. If $H\ge \btr{P}$, we further define the spacetime mean curvature of $\Sigma$ as $\mathcal{H}:=\sqrt{H^2-P^2}$ (cf. Cederbaum--Sakovich \cite{C--S}), and if $\mathcal{H}=0$ we call $\Sigma$ a \emph{generalized apparent horizon}. If $(M,g)$ indeed embedds isometrically into a spacetime $(\mathbf{M},\mathbf{g})$ with timelike, future pointing unit normal $\vec{n}$ and second fundamental form $K$, then the codimension-$2$ mean curvature vector $\vec{\mathcal{H}}$ of any orientable surface $\Sigma\subseteq M\subseteq \mathbf{M}$ can be decomposed as
		\[
		\vec{\mathcal{H}}=-H\nu+P\vec{n}.
		\]
		Note that in this case, $\theta_\pm$ describe the infinitesimal area growth of $\Sigma$ in direction of the lightrays $\vec{l}_\pm=\nu\pm\vec{n}$ by the first variation of area, so MOTS/MITS are the critical points of the area functional along the lightrays $\vec{l}_\pm=\nu\pm\vec{n}$. Moreover  length of the mean curvature vector is given by
		\begin{align*}
		\mathfrak{g}(\vec{\mathcal{H}},\vec{\mathcal{H}})=\theta_+\theta_-=\mathcal{H}^2.
		\end{align*}
		In particular, we see that a MOTS/MITS is always a generalized apparent horizon. However, the converse statement is not true in general. See \cite{Carrasco-Mars} for a counterexample.

		In this work, we are interested in considering a family of initial data sets with a particular structure.
		
		\begin{defi}\label{def-classM}
		Let $\mathcal{M}$ be the set of initial data sets $(M,g,K)$ of the form $M=(r_0,r_1)\times\bS^{n-1}$ for some $0 \leq r_0< r_1 \leq \infty$ and
		\begin{align}
		g &= N(r,\cdot)^2 dr^2 + r^2 \s(r,\cdot), \label{eq-gDef} \\
		K &= k(r,\cdot) N(r,\cdot)^2 dr^2 + p(r,\cdot) r^2 \s(r,\cdot). \label{eq-Kdef} 
		\end{align}
		where  $N \defeq N(r,\cdot)$, $k \defeq k(r,\cdot)$ and $p \defeq p(r,\cdot)$ are smooth functions on $M$ ($r_0 < r < r_1$), and $\{ \s(r,\cdot) \}_{r_0 \leq r \leq r_1}$ is a smooth family of metrics on $\bS^{n-1}$ satisfying
		\begin{align}\label{eq_volumepreserving}
			\tr_{\s(r)}\s'(r)=0,
		\end{align}
		so that the volume form induced by $\s(r)$ is constant in $r$. 
		\end{defi}
		
		\begin{bem}
		Note that when $0 < r_0 < r_1 < \infty$ in Definition~\ref{def-classM}, we can include the boundaries. In this case we can consider $M=[r_0,r_1]\times\bS^{n-1}$ as a collar extension of the surface $\S_{r_0} = \{r_0 \} \times \bS^{n-1}$, and will usually assume that the functions $N$, $p$, $k$ are smooth on $\overline{M}$ ($r_0\leq r\leq r_1$).
		\end{bem}

		The type of manifolds in the family $\mathcal{M}$ have proven to be very useful in time symmetry $(K=0)$. Mantoulidis and Schoen \cite{M--S} used manifolds with this structure to successfully compute the Bartnik mass of \emph{minimal Bartnik data} $(\S \cong \bS^2,\s,H \equiv 0)$ for metrics such that the first eigenvalue of the operator $-\Lap+K(g)$ is strictly positive (see \cite{CM,CCGP,M-Li} for extensions to higher dimensions). They further yield explicit examples for the instability of the Penrose Inequality \cite{CC-Survey}.

		By performing a change of variables from the usual spherical coordinates to rectangular coordinates $\{ x^i \}$, we see that for a member of $\mathcal{M}$ we have
		\begin{align*}
		g_{ij} &= \delta_{ij} + (N(r,\cdot)^2-1)\theta_i \theta_j + r^2 \wt_{ij} \\
		K_{ij} &= p(r,\cdot)\delta_{ij} + (k(r,\cdot)N(r,\cdot)^2-p(r,\cdot))\theta_i\theta_j + p(r,\cdot)r^2 \wt_{ij}
		\end{align*}
		where $\theta_i = \frac{x^i}{r}$ and $ \wt$ denotes the pullback of $\tau(r,\cdot) \defeq \s(r,\cdot) - \s_{\bS^{n-1}}$. In particular, we can rephrase the desired asymptotics as assumptions on $N(r,\cdot)$, $k(r,\cdot)$, $p(r,\cdot)$ and $\tau(r,\cdot)$ to make the resulting manifold asymptotically flat. For simplicity we will often write $\s(r)$ instead of $\s(r,\cdot)$, and we will avoid to write the explicit dependences of $N$, $k$, and $p$ when it is notationally convenient. We will mainly work with a fixed set of coordinates on $\bS^{n-1}$ denoted by indices in capital letters which range from $1$ to $n-1$. Additionally, we identify the index 0 with the $r$-coordinate.
		
		Due to the above considerations and motivated by the applications we have in mind, we consider the subset $\mathcal{M}_{\AF}$ of $\mathcal{M}$ (with $r_1=\infty$) consisting of asymptotically flat initial data sets $(M,g,K)$ with the additional assumption that the family of metrics $\{ \s(r,\cdot) \}_{r_0 \leq r < \infty}$  satisfies
		\begin{align}\label{eq_exponentialconvergence}
			\s(r,\cdot) \to \s_{\bS^{n-1}}\text{ exponentially fast as }r \to \infty,
		\end{align}
		where $\s_{\bS^{n-1}}$ denotes the standard round metric on $\bS^{n-1}$.
		 More precisely, we say that $M\in\mathcal{M}$ with $r_1= \infty$  is \emph{asymptotically flat} or equivalently, $M\in\mathcal{M}_{\AF}$, if $\bmu$ and $\bJ$ are integrable and the following decay conditions are satisfied:		
		\begin{enumerate}
			\item[(i)] $r^{a}|N-1| + r^{a+1}|\pr_r N|+r^{a+1}|\pr_I N| + r^{a+2}|\pr_{IJ} N| \leq C$,
			\item[(ii)] $r^{b}|p| + r^{b+1}|\partial_r p|+ r^{b+1}|\partial_Ip| \leq C$,
			\item[(iii)] $r^{b}|k| + r^{b+1}|\partial_Ik| \leq C$, and
			\item[(iv)] $\tau_{IJ} \to 0$ exponentially fast as $r\to \infty$,
		\end{enumerate}
		for $a>\frac{n-2}{2}$,$b>\frac{n}{2}$, and $C>0$  constant, where $\tau_{IJ}$ denotes the local expression of the tensor $\tau \defeq \s(r,\cdot) - \s_{\bS^{n-1}}$.
		
		We remark that this notion of asymptotic flatness, i.e., $M\in\mathcal{M}_{AF}$, is not fully equivalent to the general notion of asymptotic flatness in Definition~\ref{defi_asymflat}, as we only impose decay on the first derivatives of $N$ in radial direction in (i), no decay on the radial derivative of $k$ in (iii), and assume that $\tau$ decays exponentially (iv). However, these assumptions are natural in the context of this work (see Lemma~\ref{lem_constraints} below). See also \cite{Bartnik-QS, Lin, Jang} for similar assumptions in time-symmetry ($K=0$), in particular utilizing the exponential decay to the round metric. Moreover, assumption (iv) can be naturally achieved by evolution of the surfaces along a geometric flow, e.g. Hamilton's modified Ricci flow \cite{Hamilton2} for $n=3$. It is also possible to modify a given geometric flow to achieve condition (iv) for $n \geq 3$ \cite{CM}. This condition has been assumed implicitly or explicitly by Lin \cite{Lin}, Jang \cite{Jang}, and Mantoulidis--Schoen in \cite{M--S} and subsequent works \cite{CC-Survey}. For $M\in\mathcal{M}_{AF}$, a direct computation shows that the ADM energy $E_{\ADM}$ and ADM momentum $P$ can be expressed as
			\begin{align}
			E_{\ADM} =&\,\frac{1}{2\w_{n-1}}\lm{r}{\infty} \int_{\bS^{n-1}}\(N^2-1\) r^{n-2}\, dV_{\bS^{n-1}}, \label{eq-E-ADM}\\ 
			\textbf{P}_i=&\frac{1}{\omega_{n-1}}\lm{r}{\infty}\int_{\bS^{n-1}} pN^2r^{n-1}\nu^{i}\, dV_{\bS^{n-1}}.\label{eq-P-ADM}
			\end{align}
		Note that $E_{\ADM}$ and $\textbf{P}$ are well-defined and finite also under the asymptotics assumed for $M\in\mathcal{M}_{AF}$, as we still impose the integrability of $\bmu$, $\bJ$.
		
		As part of our aim is to extend both the work of Lin \cite{Lin} in the asymptotically flat case and the work of Jang \cite{Jang} in the asymptotically hyperbolic case, we also briefly introduce a notion of asymptotic hyperbolicity for $\Lambda<0$. For simplicity, we will usually assume $\Lambda=-\frac{(n-1)(n-2)}{2}$. A precise definition of asymptotic hyperbolicity, in particular concerning definitions of energy, mass, and momentum, is far more subtle compared to the asymptotically case, cf. \cite{CH,Wang}. Due to the symmetry of the initial data sets in $\mathcal{M}$ under consideration, we give a simplified definition of asymptotic hyperbolicity that is sufficient for our purposes. Note that apart from the differences already remarked upon in the asymptotically flat case, this notion agrees with the classical definition of asymptotic hyperbolicity for a large class of physically relevant examples in spherical symmetry such as Schwarzschild Anti de\,Sitter. Here, we say that $M\in \mathcal{M}$ is \emph{asymptotically hyperbolic} if the initial data set $(\widetilde{M}, \widetilde{g},\widetilde{K})$ in $\mathcal{M}$ determined by the functions $\widetilde{N}:=N\sqrt{1+r^2}$, $\widetilde{k}:=k$, $\widetilde{p}:=p$ and the same familiy of metrics $\{\s(r)\}$ is asymptotically flat in the sense that $\widetilde{M}\in\mathcal{M}_{AF}$. In fact, under this assumptions we can extract the total mass as defined for an asymptotically hyperbolic time symmetric initial data set as the limit of the Hawking energy, see Section~\ref{subsec_hyperboloid}. We will further discuss a notion of asymptotically hyperboloidal initial data sets in the case of $\Lambda=0$ in Section~\ref{subsec_hyperboloid}.
		
		We close this section by showing that one can rewrite the constraint equations \eqref{eq-CC-ham}-\eqref{eq-CC-mom} as a PDE system for initial data sets $(M,g,K)$ within the family $\mathcal{M}$. In the time-symmetric case, i.e., $K\equiv 0$, and when $\s(r,\cdot)=\s_{\bS^2}$, the resulting single parabolic equation for $N$ has been first studied by Bartnik in \cite{Bartnik-QS}. Further treatment of this equation in the context of general relativity under different assumptions (typically fixing a family $\s(r,\cdot)$) has been done by Smith--Weinstein \cite{SW1,SW2}, Shi--Tam \cite{S-T}, Lin \cite{Lin} and Jang \cite{Jang}. When time symmetry is not assumed, the full PDE system in the case of general initial data sets has been studied by R\'acz \cite{Racz} recasting it as a hyperbolic-parabolic system for which he shows short-time existence of a solution. A numerical approach for the spherically symmetric and near Schwarzschild case has been developed by  by Csuk\'as and R\'acz in \cite{CR_num}. For further analysis of R\'acz' approach and modifications to it see the work by Beyer, Frauendiener and Ritchie \cite{Beyer} and references therein.
		\begin{lem}\label{lem_constraints}
	Let $(M,g,K)\in\mathcal{M}$ be an initial data set with energy density $\bmu$ and the momentum density $\bJ$. Then, the constraint equations \eqref{eq-CC-ham}-\eqref{eq-CC-mom} are equivalent to the following set of equations:
			\begin{align}
			\begin{split}\label{eq_hamiltionianPDE}
			\frac{2(n-1)}{r} \pr_r N=&\,\frac{2N^2}{r^2}\Lap_{\s(r)}N  - \frac{R(\s(r))}{r^2}N^3  + \frac{(n-1)(n-2)}{r^2}N + \frac{N}{4}|\s'(r)|^2_{\s(r)}  \\
			&\, +\( - 2(n-1)k(r,\cdot)p(r,\cdot) - (n-1)(n-2)p(r,\cdot)^2\)N^3 \\
			&\,+ \( 16\pi\bmu+ 2\Lambda\)N^3,
			\end{split}\\
			(n-1)\pr_r p(r,\cdot)&=(n-1)r^{-1}\left(k(r,\cdot)-p(r,\cdot)\right)  -8\pi \bJ_0, \label{eq_momentumPDE1}\\
			0 &=\left(k(r,\cdot)-p(r,\cdot)\right)N^{-1}\dds{N}{x^I}  -(n-2) \pr_Ip(r,\cdot)-\pr_I k(r,\cdot) -  8 \pi \bJ_I.\label{eq_momentumPDE2}
			\end{align}
		\end{lem}
		\begin{proof}
			Direct computations show that
			\begin{align*}
			\tr_g K &=k(r,\cdot) + p(r,\cdot)(n-1),
			\end{align*}
			and hence,
			\begin{equation}
			\dv_{g}(K - ((\tr_{g} K) g) )_0 =
			(n-1)r^{-1}k(r,\cdot) -(n-1)r^{-1}p(r,\cdot) -(n-1)\pr_r p(r,\cdot),
			\end{equation}
			and
			\begin{align*}
			\dv_g(K)_I &=p(r,\cdot) N^{-1}\dds{N}{x^I} - k(r,\cdot) N^{-1}\dds{N}{x^I}+ \pr_Ip(r,\cdot).
			\end{align*}
			Thus the momentum constraint \eqref{eq-CC-mom} reduces to
			\begin{align*}
			(n-1)r^{-1}k(r,\cdot) -(n-1)r^{-1}p(r,\cdot) -(n-1)\pr_r p(r,\cdot) &= 8\pi \bJ_0, \\
			p(r,\cdot)N^{-1}\dds{N}{x^I} - k(r,\cdot)N^{-1}\dds{N}{x^I} -(n-2) \pr_Ip(r,\cdot)-\pr_I k(r,\cdot)  &= 8 \pi \bJ_I.
			\end{align*}
			For the Hamiltonian constraint \eqref{eq-CC-ham} we have:
			\begin{align*}
			(\tr_g K)^2
			&= k(r,\cdot)^2 + 2(n-1)k(r,\cdot)p(r,\cdot) + (n-1)^2p(r,\cdot)^2,
			\end{align*}
			and
			\begin{align*}
			|K|_{g}^2 &=k(r,\cdot)^2 + (n-1)p(r,\cdot)^2.
			\end{align*}
			The scalar curvature is given by (see, for example \cite{CC-Survey})
			\begin{align*}
			R(g)=&\,\frac{2}{r^2N}\( -\Lap_{\s(r)}N  +\frac{R(\s(t))}{2}N  \) \\
			&\,+\frac{1}{N^2}\( 2(n-1) \frac{1}{r} \frac{\pr_r N}{N} - \frac{(n-1)(n-2)}{r^2} - \frac{1}{4}|\s'(r)|^2_{\s(r)}   \).
			\end{align*}
			
			Thus the Hamiltonian constraint \eqref{eq-CC-ham} is then given by
			\begin{align*}
			0=&\,R(g) + (\tr_{g} K)^2 - |K|^2_{\g} -16\pi\bmu + 2\Lambda \\
			=&\,-\frac{2}{r^2 N}\Lap_{\s(r)}N  + \frac{R(\s(t))}{r^2} + \frac{1}{N^3}\frac{2(n-1)}{r} \pr_r N - \frac{(n-1)(n-2)}{N^2 r^2} - \frac{1}{4N^2}|\s'(r)|^2_{\s(r)}  \\
			&\,+ 2(n-1)kp + (n-1)^2p^2  - (n-1)p^2-16\pi\bmu -2\Lambda.
			\end{align*}
		\end{proof}

	\section{Penrose-like energy inequalities}\label{sec_penrose}
		In the setting of a $3$-dimensional asymptotically flat initial data set $(M,g,K)$, the Penrose conjecture states that if $(M,g,K)$ satisfies the (DEC) and is geodesically complete up to an inner boundary given as an outermost apparent horizon $\Sigma_{H}$, then
		\begin{align}\label{eq_Penrose}
		m_{\ADM} \geq \sqrt{\frac{\btr{\Sigma_{{H}}}}{16\pi}},
		\end{align}
		and equality should hold if and only if $(M,g,K)$ embeds isometrically into the Schwarzschild spacetime of appropriate mass. The boundary is assumed to be outermost in the sense that there are no other apparant horizons in $(M,g,K)$ that enclose the inner boundary $\Sigma_{H}$. It has been proven in time symmetry ($K\equiv0$) by Huisken--Ilmanen \cite{H-I} when the horizon is a connected minimal surface and by Bray for multiple components \cite{Bray}. In spherical symmetry (for $n \geq 3$) there has been several approaches as described in a survey by Mars \cite{Mars} (see also \cite{Lee-GRbook}), where the inner boundary is given by a marginally inner/outer trapped surface as a widely agreed upon notion of apparent horizon, and outermost in the sense that in either case there are no MOTS or MITS enclosing the boundary. In fact in spherical symmetry, one merely requires that the horizon is given by a generalized apparent horizon and is outermost in the sense that no leave of the canonical foliation by round spheres is a generalized apparent horizon. The two notions of outermost horizons are indeed equivalent in spherical symmetry, see Remark~\ref{rmk-untrapped}. However, we would like to point out that a Penrose Inequality for generalized apparent horizons seems to be false in general, cf. \cite{Carrasco-Mars}. Nonetheless, for initial data sets $(M,g,K)\in\mathcal{M}_{AF}$, we are able to establish Penrose-like energy inequalities with rigidity, by proving the monotonicity of the Hawking energy, where the ADM mass is replaced by the ADM energy and the boundary is given by a generalized apparent horizon. The fact that the statement remains true for generalized apparent horizons seems to be an artifact of the spherically symmetric case which sits naturally in the family $\mathcal{M}$. As a special case, we recover the full Penrose Inequality in the case of spherically symmetric and totally umbilic initial data sets, and propose a refinement of the inequality in the general case that could potentially lead to a full Penrose Inequality in this class of initial data sets.
		
		Let $(\S,\g)$ be an orientable, compact hypersurface in an initial data set $(M^n,g,K)$. When $n=3$, the Hawking energy of $\S$ is defined as
		\begin{align}\label{eq_hawkingmass}
		m_H(\S) \defeq  \sqrt{\frac{|\S|_{\g}}{16 \pi}}\( 1 - \frac{1}{16\pi}\int_{\S} \mathcal{H}^2 \, dV_{\g}     \),
		\end{align} 
		where $dV_{\g}$ is the volume form of $\g$. Note that for $(\S_r \defeq \{r \} \times \bS^{2},r^2 \s(r))$ in $(M,g,K)\in \mathcal{M}_{AF}$ we have
		\begin{align*}
		\theta_{\pm}  = \frac{2}{rN(r,\cdot)}\pm 2p(r,\cdot).
		\end{align*}
		Since $|\S_r| = 4\pi r^{2} $ (by (ii)), the Hawking energy of $\S_r$ is
		\begin{equation*}
		m_H(\S_r) = \frac{1}{8 \pi} \int_{\bS^2} r\( 1 + r^2 p^2 - \frac{1}{N^2}  \)\, dA_{\s(r)}.
		\end{equation*}
		
		\begin{defi} \label{def-untrapped}
			Let $(M,g,K)$ be an initial data set and let $\{ S_r \}_{r \in I}$ be a family of orientable, closed hypersurfaces in $M$ for some intervall $I\subseteq \R$. We say that the family $\{ S_r \}_{r \in I}$ satisfies the untrapped condition if
			\begin{equation*}
			0 \leq \mathcal{H}^2 = H^2 - P^2 
			\end{equation*}
			for all $S_r$.
			When the inequality holds strictly, we say that it satisfies the strictly untrapped condition.
		\end{defi}
			\begin{bem} \label{rmk-untrapped}
				If the initial data set isometrically embeds into an ambient spacetime, then the strictly untrapped  condition agrees with the physical notion of untrapped surfaces in general relativity. Note that if $(M,g,K)$ admits a foliation satisfying the strictly untrapped condition up to the inner boundary, then $(M,g,K)$ does not contain any closed generalized apparent horizons (enclosing the boundary), since there has to be at least one point along the surface where any such surface $\Sigma$ touches a leave $S_{r}$ tangentially (locally) from the inside, implying $H>P$ at this point. In particular, if the boundary is a generalized apparent it is outermost. In spherical symmetry, the converse is also true by the intermediate value theorem.
		\end{bem}
		For $(M,g,K) \in \mathcal{M}_{\AF}$, the foliation of canonical leaves $\{ \S_r \}_{r \in (r_0,\infty)}$ with $\Sigma_r=\{r\}\times \Sbb^2$ satisfies the strictly untrapped condition iff
		\begin{align}\label{eq_strictlyuntrapped}
		\frac{4}{r^2}\(\frac{1}{N(r,\cdot)^2} - r^2 p(r,\cdot)  \) >0\text{ for all }r>r_0.
		\end{align}
		We recall the well-known fact that the limit of the Hawking energy of $\S_r$ for $r \to \infty$ is the ADM energy, as stated in the following lemma for the readers convenience.
		\begin{lem} \label{lemma-HtoE-ADM}
			Let $(M=[r_0,\infty)\times \bS^2,g,K) \in \mathcal{M}_{\AF}$ and $\S_r=\{r\} \times \bS^2$. Then,
			\begin{equation*}
			\lm{r}{\infty} m_H(\S_r) = E_{\ADM}.
			\end{equation*}
		\end{lem}
		\begin{proof}
			Recall that by definition (see \eqref{eq-E-ADM})
			\begin{equation*}
			E_{\ADM} =\frac{1}{8 \pi}\lm{s}{\infty} \int_{\bS^2} r\(N^2-1\)  dV_{\bS^2}.
			\end{equation*}
			On the other hand,
			\begin{align*}
				\lm{r}{\infty} m_H(\S_r) &= \lm{r}{\infty}  \left[ \frac{1}{8 \pi} \int_{\bS^2} r\( 1 + r^2 p^2 - \frac{1}{N^2}  \)\, dA_{\s(r)}  \right] \\
				&= \frac{1}{8 \pi} \lm{r}{\infty}  \int_{\bS^2} r\(  \frac{N^2 - 1}{N^2}  \)\, dA_{\s(r)} +  \frac{1}{8 \pi}\lm{r}{\infty}  \int_{\bS^2}  r^3 p^2  \, dA_{\s(r)}
			\end{align*}
			and the result follows by recalling the decay conditions (1)-(3), which imply that $N \to 1$ and the second term disappears as $r \to \infty$.
		\end{proof}

		When $n=3$, we are able to guarantee the monotonicity of the Hawking energy along $\{ \S_r \}$ in $(M,g,K) \in \mathcal{M}_{\AF}$ when the dominant energy condition holds, which in particular leads to a \emph{Penrose-type energy inequality}.
		
		\begin{lem} \label{lemma-monH}
			Let $(M=[r_0,\infty)\times \bS^2,g,K) \in \mathcal{M}_{\AF}$ and $\{ \S_r \}_{r \in [r_0,\infty)} \subset M$ be the foliation of canonical leaves as above. Suppose that the inner boundary, $\S_{r_0}$, is a generalized apparent horizon and that the family $\{ \S_r \}_{r \in (r_0,\infty)}$ satifies the strictly untrapped condition. If the dominant energy condition \eqref{eq_DEC} holds, then 
			\[
				\textstyle{\ddt{}{r}m_H(\S_r) \geq 0}.
			\] 
			Moreover, the following Penrose-like energy inequality holds:
			\begin{equation} \label{eq-PTineq}
			E_{\ADM} \geq  \sqrt{\frac{|\S_{r_0}|}{16 \pi}}.
			\end{equation}
		\end{lem}
		
		\begin{proof}
			A direct computation yields
			\begin{align*}
			8\pi \textstyle{\ddt{}{r}} m_H(\S_r)&= \int_{\bS^2} \( 1+r^2p^2-\frac{1}{N^2}\)\,dA_{\s(r)}  \\
			&\qquad + \int_{\bS^2} r\left[2rp^2+2r^2p\pr_r p + \frac{2}{N^3}\pr_r N   \right] \, dA_{\s(r)} \\
			&=\int_{\bS^2}\( 1-K(\s(r))  \) \,dA_{\s(r)} +\int_{\bS^2}  \frac{1}{N^2} \la \nabla^{\s(r)} N,\nabla^{\s(r)} N \ra\,dA_{\s(r)} \\
			&\qquad + \int_{\bS^2}   \frac{r^2}{8N^2}|\s'(r)|^2_{\s(r)}  \,dA_{\s(r)} +\int_{\bS^2} 8\pi r^2\( \bmu - rp \bJ_0 \)\,dA_{s(r)} \\
			&\geq \int_{\bS^2} 8\pi r^2\( \bmu - rp \bJ_0 \)\,dA_{\s(r)},
			\end{align*}
			where $K(\sigma(r))$ denotes the Gauß curvature of $\sigma(r)$, and we used Lemma~\ref{lem_constraints} \eqref{eq_hamiltionianPDE} and \eqref{eq_momentumPDE1} for $n=3$ and $\Lambda=0$ in the second to last line and the Gauß--Bonnet Theorem in the last line.
			Recall that from the strictly untrapped condition \eqref{eq_strictlyuntrapped} we have $N^{-1}>r|p|$, hence
			\begin{align*}
			8\pi \textstyle{\ddt{}{r}} m_H(r,\cdot)& \geq  \int_{\bS^2} 8\pi r^2\( \bmu - N^{-1} |\bJ_0| \)\,dA_{\s(r)} \\
			&\geq  \int_{\bS^2} 8\pi r^2\( \bmu - |\bJ|_{g} \)\,dA_{\s(r)},
			\end{align*}
			where we have used $|\bJ|^2_g = N^{-2}\bJ_0^2 + r^{-2}|\bJ^T|_{\s(r)}^2$. Hence, the monotonicity follows directly from the (DEC).
			
			The Penrose-like inequality \eqref{eq-PTineq} follows readily by noting that $m_H(\S_{r_0}) =  \sqrt{\frac{|\S_{r_0}|}{16 \pi}}$ and the fact that $m_H(\S_r) \to E_{\ADM}$ as $r \to \infty$ (Lemma \ref{lemma-HtoE-ADM}).
		\end{proof}
	
		Note that the dimensional restriction to $n=3$ is only necessary to apply the Gauß-Bonnet Theorem. The monotonicity of the Hawking energy and a subsequent ADM energy estimate (with rigidity as below) can be readily established for higher dimensions in spherical symmetry. For general initial data sets in $\mathcal{M}_{AF}$ this can be done by assuming that the given familiy $\sigma(s)$ satisfies an additional integral estimate for the scalar curvature $R(\sigma(r))$ as in~\cite{CCGP}.
		
		Before discussing the rigidity, we note that it is possible to obtain refined estimates for $E_{\ADM}$ for within the family $\mathcal{M}_{\AF}$. We consider the quantity $|\bJ|_g - N^{-1}|\bJ_0|$, which can be thought of as a measure of the size of the angular contribution in $|\bJ|_g$.
		
		\begin{lem}\label{lemma-int-eta}
			Let $(M=[r_0,\infty)\times \bS^2,g,K) \in \mathcal{M}_{\AF}$ such that the dominant energy condition \eqref{eq_DEC} holds. Then 
			\begin{equation*}
			\int_{r_0}^{\infty} \int_{\bS^2} r^2 (|\bJ|_g - N^{-1}|\bJ_0|)  \, dA_{\s(r)}dr
			\end{equation*} 
			is finite.
		\end{lem}
		
		\begin{proof}
			Since $N^{-1}|\bJ_0| \leq  |\bJ|_g$, we only need to check the integrability of $r^2|\bJ|_g$. Recall that $N \to 1$ as $r\to \infty$. Let $L\ge r_0$ sufficiently big such that $2N\ge1$ for $r \geq L$. We have
			\begin{align*}
			&\int_{r_0}^{\infty} \int_{\bS^2} r^2 |\bJ|_g  \, dA_{\s(r)}\,dr \\
			=&\int_{r_0}^{L} \int_{\bS^2} r^2 |\bJ|_g  \, dA_{\s(r)}\,dr+\int_{L}^{\infty} \int_{\bS^2} r^2 |\bJ|_g  \, dA_{\s(r)}\,dr \\
			\leq&\int_{r_0}^{L} \int_{\bS^2} r^2 |\bJ|_g  \, dA_{\s(r)}\,dr+2\int_{L}^{\infty} \int_{\bS^2} r^2 N|\bJ|_g  \, dA_{\s(r)}\,dr\\
			\leq&\int_{r_0}^{L} \int_{\bS^2} r^2 |\bJ|_g  \, dA_{\s(r)}\,dr+2\int_M|\bJ|_g\d\operatorname{vol}_M,
			\end{align*}
			which is finite by the integrability of $|\bJ|_g$ on $M$.
		\end{proof}
		
		Motivated by the computation of the monotonicity of the Hawking energy in the proof of Lemma \ref{lemma-int-eta},  we define
		\begin{align}\label{eq_ModifiedHawkingmass}
		\mathfrak{M}_f(r) \defeq m_H(\S_r) - \int_{r_0}^r \int_{\bS^2} s^2 f(s,\cdot) \, dA_{\s(s)}  ds,
		\end{align}
		for any non-negative function $f$ on $M$ such that $f \leq |\bJ|_g -N^{-1} |\bJ_0|$. We purposely leave some freedom to choose $f$ in the definition of $\mathfrak{M}_f$, as we will see below that there are several possible non-trivial choices that allow for a full rigidity statement (see Remark~\ref{rmk-fchoices} below), and there might be some particular geometric choices of interest. Most appealing would be a choice of $f$ in divergence from, giving rise to a notion of quasilocal momentum. Note that $\bJ_I=0$ for all $I$ in spherical symmetry, so $f\equiv 0$ and $\mathfrak{M}_f$ reduces to the Hawking energy in spherical symmetry. The next Lemma establishes a monotonicity for $\mathfrak{M}_f$ in the familiy $\mathcal{M}_{\AF}$ which in particular yields a refined energy estimate:
		
		\begin{lem} \label{lemma-PI-f}
			Let $(M=[r_0,\infty)\times \bS^2,g,K) \in \mathcal{M}_{\AF}$ and $\{ \S_r \}_{r \in [r_0,\infty)} \subset M$ be the foliation of canonical leaves as above. Suppose that the inner boundary, $\S_{r_0}$, is a generalized apparent horizon and that the family $\{ \S_r \}_{r \in (r_0,\infty)}$ satisfies the strictly untrapped condition. Suppose that the dominant energy condition \eqref{eq_DEC} holds. If $0\le f \leq |\bJ|_g -N^{-1} |\bJ_0|$, then the quantity
			\begin{equation*}
			\mathfrak{M}_f(r)=m_H(\S_r) -  \int_{r_0}^r \int_{\bS^2} s^2 f(s,\cdot) \, dA_{\s(s)}ds
			\end{equation*}
			is monotone (with respect to $r$) with $\mathfrak{M}_f(r_0) = \sqrt{\frac{|\S_{r_0}|}{16 \pi}}$. Moreover, the following Penrose-like energy inequality holds:
			\begin{equation*}
			E_{\ADM} \geq \sqrt{\frac{|\S_{r_0}|}{16 \pi}}+\int_{r_0}^{\infty} \int_{\bS^2} r^2 f(r,\cdot) \, dA_{\s(r)}dr.
			\end{equation*}
		\end{lem}
		
		\begin{proof}
			Using the calculations in the proof of Lemma \ref{lemma-monH}, we readily establish that
			\begin{align*}
			\textstyle{\ddt{}{r}}\mathfrak{M}_f(r)&=\textstyle{\ddt{}{r}}m_H(\S_r) -  \int_{\bS^2} r^2 f(r,\cdot) \, dA_{\s(r)} \\
			&\geq   \int_{\bS^2} r^2\( \bmu - N^{-1} |\bJ_0| \)\,dA_{\s(r)}- \int_{\bS^2} r^2 (|\bJ|_g - N^{-1}|\bJ_0|) \, dA_{\s(r)}  \\
			&\geq \int_{\bS^2} r^2\( \bmu - |\bJ|_g \)\,dA_{\s(r)}\\
			&\ge 0.
			\end{align*}
			Clearly $\mathfrak{M}_f(r_0) = \sqrt{\frac{|\S_{r_0}|}{16 \pi}}$, and the last statement follows by monotonicity, and the asymptotics of the Hawking energy, Lemma \ref{lemma-HtoE-ADM}.
		\end{proof}
	
		We now turn towards rigidity statements for the above Penrose-like energy inequalities.
		
		\begin{lem}\label{lemma_rigidity_intermediate}
			Assuming the same hypothesis as in Lemma \ref{lemma-PI-f}, if $\mathfrak{M}_f'(r) =0$, then $\s(r)=\s_{\bS^2}$, $N=N(r)$, $f=\bmu=|\bJ|_g$, and $(M,g,K)$ satisfies the constraint equations with $\mu=|\bJ|_g$ and $|\bJ_0|=0$.
		\end{lem}
		\begin{proof}
			Recall that
			\begin{align*}
			8\pi \textstyle{\ddt{}{r}} m_H(r,\cdot)&= \frac{1}{8\pi}\int_{\bS^2}   \frac{r^2}{8N^2}|\s'(r)|^2_{\s(r)}  \,dA_{\s(r)}+\frac{1}{8\pi}\int_{\bS^2}  \frac{1}{N^2} \btr{\nabla^{\s(r)}N}^2_{\sigma(r)} N\,dA_{\s(r)}\\ &\qquad+\int_{\bS^2}r^2\( \bmu - rp \bJ_0 \)\,dA_{s(r)} \\
			&\geq \int_{\bS^2} 8\pi r^2\( \bmu - rp \bJ_0 \)\,dA_{s(r)} \\
			&\geq \int_{\bS^2} 8\pi r^2\( \bmu - N^{-1} |\bJ_0| \)\,dA_{s(r)}.
			\end{align*}
			
			Therefore,
			\begin{align*}
			\textstyle{\ddt{}{r}}\mathfrak{M}_f(r) &=\frac{1}{8\pi}\int_{\bS^2}   \frac{r^2}{8N^2}|\s'(r)|^2_{\s(r)}  \,dA_{\s(r)} + \frac{1}{8\pi}\int_{\bS^2}  \frac{1}{N^2} \btr{\nabla^{\s(r)}N}^2_{\sigma(r)} N\,dA_{\s(r)}\\ &\qquad+\int_{\bS^2}r^2\( \bmu - rp \bJ_0 \)\,dA_{s(r)}-  \int_{\bS^2} r^2 f(r,\cdot) \, dA_{\s(r)} \\
			&\geq  \int_{\bS^2}  r^2\( \bmu - rp \bJ_0 \)\,dA_{s(r)} - \int_{\bS^2} r^2 (|\bJ|_g - N^{-1}|\bJ_0|) \, dA_{\s(r)} \\
			&\geq \int_{\bS^2} r^2(\bmu-|\bJ|_g) \, d A_{\sigma(r)} \ge 0.
			\end{align*}
			
			Thus, if equality holds, then we directly see that 
			\begin{align*}
			\bmu&=|\bJ|_g,\\ 
			rp \bJ_0&=N^{-1}|\bJ_0|,\\
			\s'(r)&=0,\\
			|\nabla^{\s(r)}N|&=0, \\
			f&=|\bJ|_g - N^{-1}|\bJ_0|.
			\end{align*} 
			Hence $\s(r)=\s_{\bS^2}$, and $N=N(r)$, and as $\btr{p}rN<1$ by the strictly untrapped condition, we further find that $|\bJ_0|=0$. In particular, $\bmu=|\bJ|_{g}=f$. 	
		\end{proof}
		
		\begin{lem}\label{lem_ridigity}
			Assuming the same hypothesis as in Lemma \ref{lemma-PI-f}, $\mathfrak{M}'_f=0$ for some non-negative function $f$ with $f \leq |\bJ|_g -N^{-1} |\bJ_0|$, and assuming further that equality in the latter implies $f\equiv0$. Then, $(M,g,K)$ is isometric to a spacelike, spherically symmetric slice of a Schwarzschild spacetime.
		\end{lem}
		
		\begin{proof}
			If $\mathfrak{M}'(r)=0$, then $N=N(r)$, $\sigma(r)=\sigma_{\bS^2}$ $\bJ_0=0$, and $\bmu=\btr{\bJ}_g=f$ by Lemma \ref{lemma_rigidity_intermediate}. By assumption $f=0$, so $\bmu=\btr{\bJ}_g=0$. Using the momentum constraint, Lemma \ref{lem_constraints} \eqref{eq_momentumPDE2}, we see that
			\[
				\partial_I p+\partial_I k=0 
			\]
			for all $I$. We conclude that $k+p$ depends only on $r$. Turning to the Hamiltonian constraint, Lemma \ref{lem_constraints} \eqref{eq_hamiltionianPDE} with $\mu=0$, we find that
			\begin{align*}
			k(r,\cdot)^2&=\frac{2}{r} \frac{\pr_r N}{N^3}+\frac{1}{r^2}- \frac{1}{r^2}\frac{1}{N^2}+k(r,\cdot)^2+ 2k(r,\cdot)p(r,\cdot) + p(r,\cdot)^2 \\
			&=\frac{4}{r} \frac{\pr_r N}{N^3}+\frac{2}{r^2}- \frac{2}{r^2}\frac{1}{N^2}+\(k(r,\cdot)+ p(r,\cdot)\)^2
			\end{align*}
			showing that $k=k(r)$, and in consequence that $p=p(r)$. That is, the initial data set is spherically symmetric and saturates the Penrose inequalityn with $E_{\ADM}=m_{\ADM}$. Thus, rigidity follows from the rigidity in spherical symmetry, cf. \cite{Mars,Lee-GRbook}.
		\end{proof}
		
		Thus, we can obtain the full rigidity statement for the refined lower energy bound estimate for all admissible $f$ that satisfy the assumption of Lemma \ref{lem_ridigity}:
		
		\begin{kor}\label{kor_ineqRigid}
			Let $(M,g,K) \in \mathcal{M}_{\AF}$ and $\{ \S_r \}_{r \in [r_0,\infty)} \subset \mathcal{M}$, where $\S_r=\{r\} \times \bS^2$. Suppose that the inner boundary, $\S_{r_0}$, is a generalized apparent horizon and that the family $\{ \S_r \}_{r \in (r_0,\infty)}$ satisfies the strictly untrapped condition. Let $0\le f\le (|\bJ|_g - N^{-1}|\bJ_0|)$ be any non-negative function, such that $f\equiv 0$ if equality holds everywhere. If the dominant energy condition \eqref{eq_DEC} holds, then
			\begin{equation} \label{eq-refinedPTineq}
			E_{\ADM} \ge  \sqrt{\frac{|\S_{r_0}|}{16 \pi}}+\int_{r_0}^{\infty} \int_{\bS^2} r^2 f(r,\cdot) \, dA_{\s(r)}dr,
			\end{equation}
			and equality holds if and only if $(M,g,K)$ is isometric to a spacelike, spherically symmetric slice of a Schwarzschild spacetime.
		\end{kor}

\begin{bem}\label{rmk-fchoices}
		Notice that the condition on $f$ above is not difficult to achieve. We give two examples:
		\begin{enumerate}
			\item Let $1>\veps>0$, then take
			\begin{equation*}
			f=(1-\veps)\(|\bJ|_g - N^{-1}|\bJ_0|\).
			\end{equation*}
			\item Consider
			\begin{equation*}
			|\bJ|_g - N^{-1}|\bJ_0| =\frac{|\bJ|_g^2 - N^{-2}|\bJ_0|^2 }{|\bJ|_g + N^{-1}|\bJ_0| }=\frac{r^{-2}|\bJ^{T}|_{\s(r)}^2  }{|\bJ|_g + N^{-1}|\bJ_0|} \geq  \frac{r^{-2}|\bJ^{T}|_{\s(r)}^2  }{2|\bJ|_g}=f.
			\end{equation*}
		\end{enumerate}
\end{bem}

		Furthermore, $f\equiv0$ is always admissible and we recover the full Penrose Inequality in spherical symmetry as a special case of Corollary \ref{kor_ineqRigid}. In fact, Corollary \ref{kor_ineqRigid} with $f=0$ slightly extends the Penrose Inequality outside of spherical symmetry, as the ADM momentum $\mathbf{P}$ also vanishes for initial data sets in $\mathcal{M}_{AF}$ that are merely spherical symmetric outside a large compact set. Moreover, as $(|\bJ|_g - N^{-1}|\bJ_0|)\not=0$ in the large compact set in general, this even yields a refined lower bound on the ADM mass in this case. Altogether, this observation motivates the following conjecture:
		\begin{conj}\label{conj1}
		Let $(M,g,K) \in \mathcal{M}_{\AF}$ and $\{ \S_r \}_{r \in [r_0,\infty)} \subset \mathcal{M}$, where $\S_r=\{r\} \times \bS^2$. Suppose that the inner boundary, $\S_{r_0}$, is a generalized apparent horizon and that the family $\{ \S_r \}_{r \in (r_0,\infty)}$ satisfies the strictly untrapped condition. Then,
			\[
			|\mathbf{P}|\le \int_{r_0}^{\infty} \int_{\bS^2} r^2 (|\bJ|_g - N^{-1}|\bJ_0|)  \, dA_{\s(r)}dr.
			\]

		\end{conj}
		Conjecture \ref{conj1} suggests that $\bJ_I=0$ would imply $\mathbf{P}=0$. This seems to be consistent with our construction below in Section \ref{sec_solving}. Notice that Conjecture \ref{conj1} would imply the Penrose Inequality with rigidity, as it would follow that
		\begin{align*}
		\sqrt{\frac{|\S_{r_0}|}{16 \pi}}&\le E_{\ADM}-\int_{r_0}^{\infty} \int_{\bS^2} r^2 (|\bJ|_g - N^{-1}|\bJ_0|)  \, dA_{\s(r)}dr\\
		&\le E_{\ADM}-|\mathbf{P}|\\
		&\le \sqrt{E_{\ADM}^2-|\mathbf{P}|^2}\\
		&=m_{\ADM},
		\end{align*}
		where equality would imply $\int_{r_0}^{\infty} \int_{\bS^2} r^2 (|\bJ|_g - N^{-1}|\bJ_0|)  \, dA_{\s(r)}dr=|\mathbf{P}|=0$, so ridigity would follow from Lemma \ref{lem_ridigity}.
		
		Finally, we consider the totally umbilic case, i.e., $K=pg$ with $k=p$, where the rigidity statement follows without the above restrictions on the choice of $f$.
		
		\begin{kor} \label{coro-p-equals-k}
			Let $(M,g,K) \in \mathcal{M}_{\AF}$ and $\{ \S_r \}_{r \in [r_0,\infty)} \subset \mathcal{M}$, where $\S_r=\{r\} \times \bS^2$. Suppose that the inner boundary, $\S_{r_0}$, is a generalized apparent horizon and that the family $\{ \S_r \}_{r \in (r_0,\infty)}$ satisfies the strictly untrapped condition. Suppose that the dominant energy condition \eqref{eq_DEC} holds and suppose that $f \leq |\bJ|_g -N^{-1} |\bJ_0|$. If in addition $p=k$, then the equality 
			\begin{equation*}
			E_{\ADM} =\sqrt{\frac{|\S_{r_0}|}{16 \pi}}+\int_{r_0}^{\infty} \int_{\bS^2} r^2 f(r,\cdot) \, dA_{\s(r)}dr
			\end{equation*}
			implies $k=p=0$ (in particular $K \equiv 0$) and that $(M,g)$ is isometric to the (time-symmetric) Schwarzschild manifold outside its horizon.
		\end{kor}
		\begin{proof}
			Suppose that $k=p$ and that equality holds. By Lemma \ref{lem_ridigity} we know that $\bmu=|\bJ|_g$ with $|\bJ_0|=0$, so we can conclude that $p$ is independent from $r$ as the momentum constraint, as Lemma \ref{lem_constraints} \eqref{eq_momentumPDE1} reduces to
			\[
				2\partial_rp=0.
			\]
			Since $(M,g,K)$ is asymptotically flat, we necessarily have $p=k=0$ everywhere. Thus $\bmu=|\bJ|_g=0$ and we can conclude that we can embed $(M,g,K)$ isometrically as a spacelike, spherically symmetric slice into the Schwarzschild manifold. Since $K\equiv0$, $(M,g,0)$ is isometric to the (time-symmetric) exterior Schwarzschild manifold.
		\end{proof}
		
	\subsection{The asymptotically hyperbolic and hyperboloidal case}\label{subsec_hyperboloid}
		We can similarly discuss the monotonicity of the Hawking energy within the family $\mathcal{M}$ in the asymptotically hyperbolic case with $\Lambda=-\frac{(n-2)(n-1)}{2}<0$. For $n=3$ we have $\Lambda=-6$, and in this case, the Hawking energy is defined as
		\begin{align}
			m_H^{AH}(\Sigma)=\sqrt{\frac{\btr{\Sigma}}{16\pi}}\left(1-\frac{1}{16\pi}\int H^2-P^2\d A+\frac{\btr{\Sigma}}{4\pi}\right).
		\end{align}
		Note that the monotonicity in the family $\mathcal{M}$ follows essentially from the same computations as above (see also \cite{CCGP}). Note that as computed in \cite{miaotamxie} in a more general setting, we can also reconfirm that
		\begin{align}\label{eq_HawkinglimitAHcase}
			\lim_{r\to\infty}m_H^{AH}(\Sigma_r)=E_{AH},
		\end{align}
		under our assumptions similar to Lemma \ref{lemma-HtoE-ADM}, as the decay assumptions on $K$ ensure that the limit does not depend on $P$, and hence agrees with the limit in time-symmetry. 
		
		Moreover, we note  that in both the asymptotically flat and the asymptotically hyperbolic case the monotonicity does not depend on the asymptotic assumptions on $(M,g,K)$, which instead ensure that one can extract a well-defined limit at infinity. Here, we are particularly interested in the totally umbilic case due to its rigidity properties in the context of a Penrose-type inequality (Corollary \ref{coro-p-equals-k}) and the fact that in this case the constraints decouple and we can solve them as an evolutionary systems, see Section \ref{subsec_umbilic} below. Thus, for the case $\Lambda=0$ it is also natural to consider asymptotically hyperboloidal initial data sets $(M,g,K)$, where we assume that $(M,g)$ is asymptotically hyperbolic and $K\to g_{\mathbb{H}}$ at an approriate rate. More precisely, here we say say an initial data set $(M,g,K)\in\mathcal{M}$ is asymptotically hyperboloidal if $(M,g,K-g_{\mathbb{H}})\in \mathcal{M}$ is asymptotically hyperbolic as defined in Section \ref{sec_prelim}.
		
		Indeed, such asymptotically hyperboloidal initial data sets arise naturally as totally umbilic slices in the spherical symmetric case in a large class of spacetimes. In fact, in these examples they have constant umbilicity factor, i.e., $p=k=\operatorname{const}$. For details, we refer to work by the second named author \cite{wolff}.
		
		If we impose enough decay on $K$ for an asymptotically hyperboloidal initial data set $(M,g,K)$ such that
		\begin{align}\label{eq_moredecayP}
			p^2=1+\widetilde{p}^2,
		\end{align}
		where $\widetilde{p}\in O(r^{b})$, $b>\frac{n}{2}$, then
		\[
			m_H(\Sigma_r)=m_H^{AH}(\Sigma_r),
		\]
		where $m_H^{AH}$ is computed with respect to the asymptotically hyperbolic initial data set $(M,g,\widetilde{p}g)$. In particular, the Hawking energy (as defined in \eqref{eq_hawkingmass}) of the familiy $\{ \Sigma_r \}$ in $(M,g,K)$ converges to the total energy of the asymptotically hyperbolic initial data set $(M,g,\widetilde{p}g)$. These observations are consistent with the work of Chrusciel--Todd \cite{C-T} on the correspondence between positive constant mean curvature initial data sets with $\Lambda=0$ and maximal initial data sets with $\Lambda<0$.
		
		Moreover, under the decay assumption \eqref{eq_moredecayP} this duality further extends to solutions of the constraint equations as an evolutionary system, at least in the cases discussed below in Section \ref{sec_solving}. That is, if we impose enough decay on the initial conditions of the momentum constraint, that is on $k$ and $\bJ$ in the spherically symmetric case, and on $\bJ_0$ in the totally umbilic case, respectively, then we may solve for $p$ such that it will satisfy \eqref{eq_moredecayP}. Then, in the case of $\Lambda=0$ and \eqref{eq_moredecayP}, we may solve the Hamiltonian constraint as in the case for $\Lambda=-\frac{(n-1)(n-2)}{2}$, cf. Proposition \ref{prop_sphericalsymmetry} and Theorem \ref{thm-existence-totally-umbilical}. Note that in particular, the solution $N$ of the Hamiltonian constraint is the same, and hence so is the resulting the manifold $(M,g)$.

		\section{Solving the constraint equations}\label{sec_solving}
		In this section, we discuss the solvability of the constraint equations \eqref{eq-CC-ham}-\eqref{eq-CC-mom} as an evolutionary system. We recall that by Lemma \ref{lem_constraints}, the constraint equations \eqref{eq-CC-ham}-\eqref{eq-CC-mom} for initial data sets in $\mathcal{M}$, assuming the family of metric $\{\s(s)\}$ to be given, can be equivalently written as system of a parabolic equation for the lapse $N$ and an ODE on $p$ given by
		\begin{align}
			\begin{split}\label{eq_hamiltionianPDE_solve}
			\frac{2(n-1)}{r} \pr_r N&=\frac{2N^2}{r^2}\Lap_{\s(r)}N  - \frac{R(\s(r))}{r^2}N^3  + \frac{(n-1)(n-2)}{r^2}N + \frac{N}{4}|\s'(r)|^2_{\s(r)}  \\
			&\quad +\( - 2(n-1)k(r,\cdot)p(r,\cdot) - (n-1)(n-2)p(r,\cdot)^2\)N^3 \\
			&\qquad\qquad + \( 16\pi\bmu + 2\Lambda\)N^3,
			\end{split}\\
			(n-1)\pr_r p(r,\cdot)&=(n-1)r^{-1}\left(k(r,\cdot)-p(r,\cdot)\right)  -8\pi \bJ_0, \label{eq_momentumPDE1_solve},
		\end{align}
		which are coupled by the angular components of the Hamiltonian constraint
		\begin{align}
			0 &=\left(k(r,\cdot)-p(r,\cdot)\right)N^{-1}\dds{N}{x^I}  -(n-2) \pr_Ip(r,\cdot)-\pr_I k(r,\cdot) -  8 \pi \bJ_I,\label{eq_momentumPDE2_solve}
		\end{align}
		where the energy density $\bmu$, momentum density $\bJ$ and the extrinsic curvature component $k$ are the prescribed data and we solve for $N$ and $p$. The idea to rewrite the constraints as an evolutionary system was first studied by Bartnik \cite{Bartnik-QS} in the time-symmetric case using a quasi-spherical construction -- later Sharples \cite{Sharples} studied the non time-symmetric case for quasi-spherical metrics obtaining short time existence via an iteration scheme. The general coupled system in the form \eqref{eq_hamiltionianPDE_solve}-\eqref{eq_momentumPDE2_solve} was obtained by Racz in \cite{Racz}, where he shows short time existence relying on results of Bartnik \cite{Bartnik-QS}.
		
		Here, taking advantage of the geometric features of the family $\mathcal{M}_{\AF}$ we analytically solve the evolutionary system \eqref{eq_hamiltionianPDE_solve}-\eqref{eq_momentumPDE2_solve} in spherical symmetry and in the totally umbilic case ($p=k$). In both cases, the system decouples and we can first independently solve for $p$ and then solve for $N$. Our construction is not restricted to $n=3$ and we pose sufficient and necessary conditions for the resulting initial data set to be flat when $\Lambda=0$. We also briefly discuss the solvability in the asymptotically hyperbolic case (when $\Lambda \le 0$). As outlined in Section \ref{subsec_hyperboloid} the approach for the asymptotically hyperbolic case will also yield a construction for the closely related asymptotically hyperboloidal case, when $\Lambda=0$ but we assume that $k\to 1\not=0$ at an appropriate rate, cf. Equation \eqref{eq_moredecayP}. 
		
		Given an asymptotically flat initial data set $(M,g,K)\in \mathcal{M}_{AF}$, we see that the asymptotic conditions on $N$, $p$ and $k$ yield that $\bmu\in O(r^{-c})$, $\bJ\in O(r^{-b-1})$ for $c:=\min(2+a,2b)$, and we note that we may always assume that $b=a+1$ for convenience without loss of generality.

		Recall that we additionally require that $\bmu$, $\bJ$ are integrable on $(M,g)$. By the assymptotic decay of $N$ and the fact that the volume element is preserved along the family $\sigma(r)$ tending to the round sphere, we see that the integrability of $\bmu$ and $\bJ$ is equivalent to 
			\begin{align}
				\int_{r_0}^\infty &\int_{\bS^{n-1}} r^{n-1}\bmu \, dA_{\bS^{n-1}}  dr<\infty,\label{eq_integrabilityconditionmMu}\\
				\int_{r_0}^\infty &\int_{\bS^{n-2}} r^{n-1}\sqrt{\bJ_0^2+r^{-2}\sigma_{\bS^2}^{IK}\bJ_I \bJ_K}\, dA_{\bS^{n-1}} dr<\infty\label{eq_integrabilityconditionmJ}.
			\end{align}
		In the following, we will always assume the desired decay assumptions on $\bmu$, $\bJ$ and $k$ as mentioned above. Moreover, as we are only interested in initial data sets with integrable energy and momentum densities, we will always assume  \eqref{eq_integrabilityconditionmMu}-\eqref{eq_integrabilityconditionmJ}; in fact, they are always needed in our construction.

	\subsection{The spherically symmetric case}
	
		We assume that the given familiy of metrics $\{\s(r)\}$ is given by the round metric $\s(r)=\s_{\bS^{n-1}}$ for all $r$, and that the prescribed data is given by functions $\bmu$, $\bJ_0$, $\bJ_I$, and $k$ only depending on $r$. We want to solve for spherically symmetric initial data sets $(M,g,K)$ in $\mathcal{M}$, i.e., we solve for $N=N(r)$, $p=p(r)$. Under these assumptions, we see that \eqref{eq_momentumPDE2_solve} forces $\bJ_I\equiv 0$ and Equations \eqref{eq_hamiltionianPDE_solve}-\eqref{eq_momentumPDE1_solve} decouple and reduce to the following first order ODE system:
		\begin{align*}
		(n-1)\pr_r p(r)&=-(n-1)r^{-1}p(r) + (n-1)r^{-1}k(r)  -8\pi \bJ_0, \\
		0 &=8 \pi \bJ_I, \\
		\frac{2(n-1)}{r} \pr_r N&= - \frac{(n-1)(n-2)}{r^2}N^3  + \frac{(n-1)(n-2)}{r^2}N \\
		&\quad -\(  2(n-1)k(r)p(r) + (n-1)(n-2)p(r)^2  \)N^3 \\
		&\qquad\qquad + \( 16\pi\bmu + 2\Lambda\)N^3.
		\end{align*}
		For now we will not impose any restriction to the sign of $\Lambda$ and define the auxiliary function $f:=1-\frac{2}{n(n-1)}\Lambda r^2-\frac{1}{N^2}$. Since $f'=\frac{-4}{n(n-1)}\Lambda r+\frac{2N'}{N^3}$, we see that the above ODE system becomes
		\begin{align}
			p'&=-\frac{p}{r}+\left(\frac{k}{r}-\frac{8\pi}{n-1}\bJ_0\right),\label{eq_ODEp}\\
			f'&=-\frac{(n-2)}{r}f+\left(\frac{16\pi r}{n-1}\bmu-r(2kp+(n-2)p^2)\right)\label{eq_ODEf}.
		\end{align}
		We first solve the inhomogeneous ODE for $p$ and then use the solution $p$ to solve the inhomogeneous ODE for $f$ using the method of variation of constants. The homogeneous part of \eqref{eq_ODEp} is given by 
		\[
			p_h'=-\frac{p_h}{r}
		\]
		with solutions of the form $p_h=\frac{C}{r}$ for some constant $C$. Now considering a solution $p$ of \eqref{eq_ODEp} of the form $p(r)=C(r)p_h$, we see that
		\begin{align}\label{eq_solutionP1}
			p(r)=\frac{1}{r}\left(p_0+\dint_{r_0}^{r} s\bc_1(s)\,ds\right)
		\end{align}
		where $\bc_1(r):= \frac{k(r)}{r} -\frac{8\pi}{n-1} \bJ_0$, and $p_0$ an arbitrary real constant. By setting
		\begin{align*}
		\bc_2(r)= -r\(  2k(r)p(r) + (n-2)p(r)^2  \) + \frac{r}{n-1}16\pi\bmu,
		\end{align*}
		we can similarly proceed for $f$ and obtain
		\begin{align}\label{eq_solutionF1}
			f(r)=\frac{1}{r^{n-2}}\left[f_0 + \int_{r_0}^{r} \bc_2(s)s^{n-2} \, ds  \right]
		\end{align}
		for some constant $f_0$.
		
		Let us now discuss the asymptotic behavior of the solutions and restrict our considerations to $\Lambda \le 0$ and assume that the given functions $\bmu,\bJ_0,k\colon (0,\infty)\to \R$ satisfy the above decay conditions and the integrability conditions \eqref{eq_integrabilityconditionmMu}-\eqref{eq_integrabilityconditionmJ} on $(0,\infty)$. By the asymptotic decay on $k$ and $\bJ_0$, we see that
		\[
			\dint_{r_0}^{\infty} s\bc_1(s)\,ds<\infty,
		\]
		and we can thus write the solution $p$ as 
		\[
			p(r)=\frac{1}{r}\left(\widetilde{p}_0-\dint_{r}^{\infty} s\bc_1(s)\,ds\right)
		\]
		for some constant $\widetilde{p}_0$, and we see that $p$ only satisfies the desired decay if we impose $\widetilde{p}_0=0$. Hence, the desired solution is given by
		\begin{align}\label{eq_solutionP2}
		p(r) = -\frac{1}{r} \dint_{r}^{\infty} s\bc_1(s)\,ds
		\end{align}
		without any freedom to choose the initial value at an interior boundary. Turning towards the desired decay of the metric components, we note that we achieve the desired decay for both the asymptotically flat case ($\Lambda=0$) and the asymptotically hyperbolic case ($\Lambda <0$) if $f$ decays at the same appropiate rate for both cases. More precisely, we require $f\in O_{1}(r^{-a})$\footnote{recall that we do not impose any decay on the second derivatives of $f$, as all the relevant curvature components only depend on first derivatives of $f$ in $r$, cf. the Hamiltonian constraint in Lemma \ref{lem_constraints}.}  for $a>\frac{n-2}{2}$. However, by the decay of $k$, $p$ and the integrability condition \eqref{eq_integrabilityconditionmMu}, we see that
		\[
			\int_{r_0}^{\infty} \bc_2(s)s^{n-2} \, ds<\infty,
		\]
		thus $f\le Cr^{-n+2}$. Hence, the aquired solutions $p$ and $f$ yield an asymptotically flat initial data set $(M,g,K)$ as long as the resulting metric is Riemannian up to the boundary, i.e., we choose $f_0$ in \eqref{eq_solutionF1} and $r_0>0$ such that $N>0$ on $(r_0,\infty)$ and $N(r_0)\ge 0$, where the constant $f_0$ also determines the value at the chosen boundary $r=r_0$. As we assume $\Lambda \geq 0$, this can always be done by choosing an appropriate $f_0$. We collect the above discussion in the following proposition:
		\begin{prop}\label{prop_sphericalsymmetry}
			Let $n\ge 3$, $\Lambda \le 0$ and consider smooth functions $\bmu, \bJ_0,k:(0,\infty)\to\R$ satisfying $k\in O(r^{-b})$, $\bJ_0\in O(r^{-b-1})$, and $\bmu\in O(r^{-c})$ for constants $b>\frac{n}{2}$, $c>\frac{n+2}{2}$. Assume further that $\bmu$ and $\bJ$ satisfy \eqref{eq_integrabilityconditionmMu} and \eqref{eq_integrabilityconditionmJ} on $(0,\infty)$. Let $r_0>0$, then there exists $f_0\in\R$ such that the solution $f$ of \eqref{eq_ODEf} as in \eqref{eq_solutionF1} satisfies
			\[
				1-\frac{2}{n(n-1)}\Lambda r^2\ge f
			\]
			on $[r_0,\infty)$, where the inequality is strict on $(r_0,\infty)$. \newline
			In particular, there exists a spherically symmetric initial data set $(M,g,K)$ in $\mathcal{M}_{\AF}$ with
			\begin{align*}
				M&=[r_0,\infty)\times \bS^{n-1},\\
				g&=N^2 dr^2+r^2\s_{\bS^{n-1}},\\
				K&=N^2k dr^2+pr^2\s_{\bS^{n-1}},
			\end{align*}
			where $N=\scriptsize\( \sqrt{1-\frac{2\Lambda}{n(n-1)}r^2- f} \)^{-1}$ and $p$ is the solution of \eqref{eq_ODEp} as in \eqref{eq_solutionP2}, satisfying the constraint equations \eqref{eq-CC-ham}-\eqref{eq-CC-mom} with energy density $\bmu$ and energy momentum $\bJ=\bJ_0 dr$. Additionaly, $(M,g,K)$ is asymptotically flat if $\Lambda=0$ and asymptotically hyperbolic if $\Lambda <0$.
		\end{prop}
		Note that we in fact get $f\in O(r^{-n+2})$ in Proposition~\ref{prop_sphericalsymmetry}. In particular, we see that the ADM energy is finite directly via Equation \eqref{eq-E-ADM}.
		\begin{bem}\label{bem_sphericalsymmetry}
			By choosing $r_0$ and $f_0$ appropiately, we retain the freedom to prescribe any non-negative boundary data for $N^{-1}$ at $r_0$. Note that $(N(r_0))^{-1}=0$ corresponds to a minimal boundary, and the inner boundary is given by a generalized apparent horizon if $(N(r_0))^{-1}=\btr{r_0p(r_0)}$, which is either a MOTS or a MITS depending on the sign of $p$ at $r_0$.
			
			For $\bJ_0\equiv 0$, an observation by the second named author \cite[Remark 4.2]{wolff} shows that spherical symmetric initial data sets $(M,g,K)$ as considered here embed into a static spacetime $(\mathbf{M},\mathbf{g})$ of the form
			\begin{align*}
				\mathbf{M}&=\R\times (r_1,\infty)\times \bS^{n-1},\\
				\mathbf{g}&=-h dt^2+\frac{1}{h}dr^2+r^2\s_{\bS^{n-1}},
			\end{align*}
			with $h:=\frac{1}{N^2}-r^2p^2$ and $r_1\ge r_0$ is chosen as the largest zero of $h$ (or as $0$ if $h>0$ everywhere). Note that $\{r=r_1\}$ corresponds to a Killing horizon of $(\mathbf{M},\mathbf{g})$ and the part of $(M,g,K)$ corresponding to $(r_0,r_1)$ is hidden behind the horizon. As apparent horizons arise as cross sections of such Killing horizons, we may choose $r_1=r_0$ if the inner boundary of $(M,g,K)$ is given as an outermost generalized apparent horizon. In fact, $h$ is exactly defined such that any zero of $h$ corresponds to a generalized apparent horizon in $(M,g,K)$.
			
			Hence, combining Proposition \ref{prop_sphericalsymmetry} setting $\bJ_0=0$ with the observations of the second author in \cite{wolff}, we can in fact construct an exterior region, given as a static spacetime as above, for prescribed functions $\bmu$, $k$ and a freely chosen constant $f_0$. Under the decay assumptions of Proposition \ref{prop_sphericalsymmetry}, the resulting spacetime approaches the Minkowski spacetime for $\Lambda=0$ and the Anti deSitter spacetime for $\Lambda<0$, respectively, with the same decay rate. If $h$ only posseses finitely many, simple zeros for a given choice of $f_0$, then a construction of generalized Kruskal--Szekeres coordinates by Cederbaum and the second author \cite{C-W} yields a spacetime extension of the exterior region that covers the whole range of radii $(0,\infty)$. See also \cite{brillhayward, schindagui}.
		\end{bem}
	\subsection{The totally umbilic case}\label{subsec_umbilic}
		We now consider the totally umbilic case, where we consider a given familiy of metrics $\{\s(s)\}$ satisying \eqref{eq_volumepreserving} and \eqref{eq_exponentialconvergence} and allow the described data and desired solution to depend on the radial and angular coordinates, but impose that $K=pg$, i.e., $k=p$, where we do not prescribe $k$ as initial data in this case. Similar to the spherical symmetric case, the system decouples and can be solved seperately (first solving for $p$), where we aim to construct a solution $N$ of the parabolic equation
		\begin{align}
			\begin{split}\label{eq_hamiltionianPDE_solve2}
			r\pr_r N&=\frac{N^2}{n-1}\Lap_{\s(r)}N  +\left(\frac{(n-2)}{2}+\frac{1}{8(n-1)}|\s'(r)|^2_{\s(r)}r^2\right)N\\
			&\qquad -\frac{1}{2(n-1)}\(R(\s(r))-\left(16\pi\bmu + 2\Lambda-(n-1)np^2\right)r^2\)N^3,
			\end{split}
		\end{align}
		where we multiplied \eqref{eq_hamiltionianPDE_solve} by $\frac{r^2}{2(n-1)}$ and used $p=k$, by first finding a solution $p$ of the overdetermined first order system
		\begin{align}
			\begin{split}\label{eq_momentumODE_solve2}
				\pr_r p(r,\cdot)=-4\pi \bJ_0,\\
				\pr_I p(r,\cdot)=-4\pi \bJ_I.
			\end{split}
		\end{align}
		Integrating the first equation in \eqref{eq_momentumODE_solve2} with respect to $r$ yields that any solution $p$ is of the form
		\begin{align*}
		p(r,\cdot)=C-\int_{r_0}^r4\pi \bJ_0\,ds,
		\end{align*}
		where $C=C(x^J)$ is a function independent of $r$, and it is immediate to see that $C$ in fact represents the boundary value of $p$, i.e., $p(r_0,\cdot)\equiv C$. Now taking an angular derivative of the above expression, we see that 
		\begin{align}\label{eq-umbilicalcompatibility1}
		-4\pi \bJ_I=\pr_IC-\int_{r_0}^r4\pi \pr_I\bJ_{0}\, ds.
		\end{align}
		Thus, we can only solve the first order system \eqref{eq_momentumODE_solve2} if and only if the data $\bJ$ and $C$ satisfy \eqref{eq-umbilicalcompatibility1} as a compatibility condition. We now briefly discuss the asymptotic properties of the solution $p$. By the decay assumptions on $\bJ_0$ (see Theorem \eqref{thm-existence-totally-umbilical} below), we see that
		\[
			\int_{r_0}^\infty\bJ_{0}\, ds<\infty,
		\]
		and we can rewrite $p$ as
		\begin{equation*}
		p(r,\cdot)=\widetilde{C}+\int_{r}^\infty 4\pi \bJ_0\, ds,
		\end{equation*}
		where $\widetilde{C}$ is an appropiate function independent of $r$. It is immediate to see that we require $\widetilde{C}=0$ for the desired decay of $p$. In particular, we see that
			\begin{align*}
		p(r,\cdot)=\int_{r}^{\infty}4\pi \bJ_0\,
		\end{align*}
		and
		\begin{align}
		\bJ_I(r,\cdot)= -\int_{r_0}^\infty  \pr_I\bJ_0 \, ds  + \int_{r_0}^r \pr_I\bJ_{0}\, ds =  -\int_{r}^\infty  \pr_I\bJ_0 \, ds.
		\end{align}
		Moreover, the constant of integration $C$ satisfies
		\begin{align}\label{eq-umbilicalcompatibility2}
		C=\int_{r_0}^\infty 4\pi \bJ_0 \, ds.
		\end{align}
		Hence, we can solve the momentum constraint in the totally umbilic case for $p=p(r,\cdot)$, with the appropriate decay for an asymptotically flat initial data set, if and only if the data $\bJ$ and $C$ satisfy the two compatibility conditions \eqref{eq-umbilicalcompatibility1} and \eqref{eq-umbilicalcompatibility2}, and in addition, $\bJ_0$ and its angular derivates decay sufficiently fast. In particular, we solve for the unknown $p$ by prescribing only a single function $\bJ_0$ and construct the rest of the data $\bJ$ and $C$ from $\bJ_0$.
		
		Turning towards solving Equation \eqref{eq_hamiltionianPDE_solve2}, we recall that the term
		\[ 
			16\pi\bmu +  2\Lambda-(n-1)np^2
		\]
		given by the prescribed data agrees exactly with the scalar curvature by the Hamiltonian constraint \eqref{eq-CC-ham} Thus, we may appeal to the construction of Riemannian manifolds of prescribed scalar curvature building on Bartnik's quasi-spherical metrics \cite{Bartnik-QS}. In particular, we will use the the results of Lin \cite{Lin} and Jang \cite{Jang} in the asymptotically flat and asymptotically hyperbolic case, respectively. This leads to the following theorem:
		\begin{thm} \label{thm-existence-totally-umbilical}
			Let $\{\s(s)\}$ be a family of metrics on $\bS^{n-1}$ satisfying \eqref{eq_volumepreserving} and \eqref{eq_exponentialconvergence}. Let $r_0>0$ and $\Lambda\le 0$ and consider functions $\bmu,\bJ_0\colon [r_0,\infty)\times\bS^{n-1}\to\R$ with $\bmu\in O(r^{-c})$, $\bJ_0\in O(r^{-b-1})$ for constants $c>\frac{n+2}{2}$, $b>\frac{n}{2}$. Assume further that 
			\[
				\bJ_I:=-\int_r^\infty \partial_I \bJ_0\, ds\in  O(r^{-b-1}),
			\]
		 	and that $\bmu$, $\bJ=\bJ_0dr+\bJ_I dx^I$ satisfy the integrability condition \eqref{eq_integrabilityconditionmMu}-\eqref{eq_integrabilityconditionmJ}.\newline
			Then there exists an totally umbilic initial data set $(M,g,K)$ in $\mathcal{M}$ with
			\begin{align*}
				M&=[r_0,\infty)\times\bS^{n-1},\\
				g&=N^2dr^2+r^2\s(r),\\
				K&=p g,
			\end{align*}
			solving the constraint equations \eqref{eq-CC-ham}-\eqref{eq-CC-mom}, where $N$, $p$ are solutions of \eqref{eq_hamiltionianPDE_solve2}-\eqref{eq_momentumODE_solve2}. Further, $(M,g,K)$ is asymptotically flat if $\Lambda=0$ and asymptotically hyperbolic if $\Lambda <0$, respectively.
		\end{thm}
		\begin{proof}
			We aquire a solution $p$ of \eqref{eq_momentumODE_solve2} with the desired decay as discussed above, so we are left to show that \eqref{eq_hamiltionianPDE_solve2} admits a solution with the desired decay. We define
			\[
				\overline{R}:=16\pi\bmu + 2\Lambda-(n-1)np^2.
			\]
			For $n=3$ and appropiate initial data for $N$ at $r_0$, we may directly apply the results of Lin \cite{Lin} if $\Lambda=0$ and of Jang \cite{Jang} if $\Lambda=-3$, respectively. By going over the proofs in Lin~\cite{Lin} and  Jang~\cite{Jang} of the asymptotic decay using Schauder estimates, one can check that the weaker decay assumptions here lead to the desired asymptotics, i.e., the resulting initial data set $(M,g,K)$ is asymptotically flat or asymptotically hyperbolic as defined in Section \ref{sec_prelim}.
			
			In fact, the results of Lin~\cite{Lin} and Jang~\cite{Jang} further extend to higher dimensions. We refer to Appendix \ref{appendix} for a brief overview of the necessary adjustments, in particular a derivation for the barriers in all dimensions.
		\end{proof}
		\begin{bem}
			In the asymptotically flat case, Lin~\cite{Lin} further provided a condition for admissible data for the mean curvature $H$ at the interior boundary. In particular, one can always construct initial data with minimal inner boundary. However, in the context of non-time symmetric initial data it is desirable for the boundary to be given by a generalized apparent horizon. From \cite{Lin}, we can directly construct such initial data with $p=k$, if we additionally have
			\begin{align}\label{eq_weird}
				16\pi^2\left(\int_r^\infty \bJ_0\right)^2>\sup\limits_{1\le t\le \infty}\left[\int_1^t\left(8\pi\tau^2\bmu-\frac{R(\sigma(\tau))}{2}-3\tau^2p^2\right)_*\exp\left(\int_1^\tau\frac{s\btr{\sigma'(\tau)}^{*2}}{8}  ds\right)d\tau\right],
			\end{align}
			where
			\[
			f_*(r):=\inf\limits_{S^{n-1}}f(r,\cdot)\text{, and }f^*(r):=\sup\limits_{S^{n-1}}f(r,\cdot)
			\]
			for a function $f \colon M \To \R$.
		\end{bem}

	\section{Comments and future directions}\label{sec_discussion}

The family of (non time-symmetric) initial data sets $\mathcal{M}_{\AF}$ can be thought as a \emph{slight} generalization of rotational symmetry, in the sense that, at infinity, the level sets $r=\textnormal{constant}$ are almost round near infinity. However, it is general enough to be a rich family that allows non trivial momentum vector while being computationally convenient. The following are natural future directions to pursue to exploit the structure of the initial data sets discussed in this work.

The content of Theorem~\ref{thm-existence-totally-umbilical} is a new step towards finding solutions to the constraint equations in the form of a evolutionary system suggested by R\'acz in \cite{Racz}. It is plausible that the structure imposed in $\mathcal{M}_{\AF}$ could lead to more general results that the ones presented here, i.e., when $k \neq p$. 

We have seen that the members of the class $\mathcal{M}_{AF}$ satisfy the positive Penrose-type energy inequalities with the appropriate rigidity statement as derived in Section \ref{sec_penrose}. It is thus natural to ask about the stability of these results. For example, is there any non-trivial choice of $f$ such that for an initial data set $(M,g,K)$ whose energy is close to  $\sqrt{\frac{|\S_{r_0}|}{16 \pi}}+\int_{r_0}^{\infty} \int_{\bS^2} r^2 f(r,\cdot) \, dA_{\s(r)}\,dr$,  is $(M,g,K)$ close to a spherically symmetric spacelike slice of the  Schwarzschild spacetime with mass $E_{\ADM}(M,g,K)$? The relevance and context of these type of questions and the state of the art is very well explained and motivated in a recent survey by Sormani \cite{Sormani-scalar} which we recommend to consult. In particular, we want to mention the work of Mantoulidis and Schoen \cite{M--S} and generalizations to various different settings (see \cite{CCM,ACC,CCGP})  that suggest the Penrose Inequality to be unstable. The collar extensions constructed by Mantoulidis and Schoen in \cite{M--S} are in fact members of the class $\mathcal{M}_{AF}$ in time symmetry. This suggests that the Penrose-type energy inequalities, and by virtue of their construction indeed the full Penrose inequality, should be unstable in the class $\mathcal{M}_{AF}$ by adapting their construction to the non time symmetric case.

Moreover, we note that Mantoulidis and Schoen used these collar extensions in \cite{M--S} to compute the Bartnik mass in the minimal case (see also \cite{CM}). We believe it is possible to adapt this construction using the structure in $\mathcal{M}$ to obtain initial data sets with prescribed (MOTS) boundary geometry to attempt to compute the Bartink mass (appropriately defined) for non time-symmetric initial data sets with a MOTS. If this is successful, similar considerations could be taken to extend the minimal case to the non time-symmetric case in the asymptotically hyperbolic and charged case \cite{ACC,CCM,CCGP}.

Additionally, certain choices of $f$, in particular in divergence form, might lead to a new notion of quasi-local momentum and due to the monotonicity in this special case, this possibly could lead to new considerations regarding the Penrose Inequality.

\appendix
\section{Solving the Hamiltonian constraint in higher dimensions}\label{appendix}

We briefly outline how to extend the work of Lin \cite{Lin} in the asymptotically flat case, and the work of Jang \cite{Jang} in the asmypotitcally hyperbolic case to higher dimensions. The crucial step is the derivation of the appropriate barriers. 
Recall that $N$ satisfies the equation
\begin{align}
\begin{split}\label{eq_hamiltionianPDE_appendix}
r\pr_r N&=\frac{N^2}{n-1}\Lap_{\s(r)}N  +\left(\frac{(n-2)}{2}+\frac{1}{8(n-1)}|\s'(r)|^2_{\s(r)}r^2\right)N\\
&\qquad -\frac{1}{2(n-1)}\(R(\s(r))-\overline{R}r^2\)N^3,
\end{split}
\end{align}
where
\[
	\overline{R}:=16\pi\bmu + 2\Lambda-2(n-1)kp-(n-1)(n-2)p^2,
\]
and we treat $\overline{R}$ as prescribing the desired scalar curvature on $(M,g)$. Here, we always assume that $p$ and $k$ are given functions satisfying the desired decay. Recall from Section \ref{sec_solving} that the constraint equations decouple in the spherically symmetric and the totally umbilic case within the class of initial data sets considered here, and we may first solve the momentum constraint to obtain $p$ as desired. As described in Section \ref{sec_solving} it thus remains to solve the Hamiltonian constraint as a parabolic system with prescribed scalar curvature $\overline{R}$, as in the work of Lin \cite{Lin} and Jang \cite{Jang}. Under the decay assumptions on $p$ and $k$, and the integrability condition \eqref{eq_integrabilityconditionmMu}, we note that
\begin{align}\label{eq_integrabilityR}
	\int_{r_0}\int_{\bS^{n-1}}r^{n-1}\overline{R} \, dV_{\mathbb{S}^{n-1}} dr<\infty.
\end{align}

Let us first consider the asymptotically flat case with $\Lambda=0$: We then consider the function $\omega:=N^{-2}$. Using \eqref{eq_hamiltionianPDE_appendix}, we see that $\omega$ satisfies
\begin{align}
	\begin{split}\label{eq_evolutionomega}
		\partial_r\omega=&\,\frac{1}{(n-1)r}\omega^{-1}\Delta_{\sigma(r)}\omega+\frac{3}{(n-1)}\frac{N}{r}{}^{\s(r)}\!\nabla N\cdot{}^{\s(r)}\!\nabla \omega\\&\,-\left(\frac{n-2}{r}+\frac{r}{4(n-1)}\btr{\s'(r)}^2_{\s(r)}\right)\omega
		+\frac{1}{r}\left(\frac{R(\s(r))-\overline{R}r^2}{(n-1)}\right).
	\end{split}
\end{align}
ODE comparrison yields for all $r_0\ge 1$.
\begin{align*}
N^{-2}(r,\cdot)&\ge \delta_*^n(r)+\left(\frac{r_0}{r}\right)^{n-2}\left(u^*(r_0)^{-2}-\delta_*^n(r_0)\right)\exp\left(-\int\limits_{r_0}^r\frac{s\left(\btr{\s'(s)}^{2}_{\s(s)}\right)^*}{4(n-1)} \, ds\right),\\
N^{-2}(r,\cdot)&\le \delta_n^*(r)+\left(\frac{r_0}{r}\right)^{n-2}\left(u_*(r_0)^{-2}-\delta_n^*(r_0)\right)\exp\left(-\int\limits_{r_0}^r\frac{s\left(\btr{\s'(s)}^{2}_{\s(s)}\right)_*}{4(n-1)} \, ds\right),
\end{align*}
with barriers
\begin{align*}
\delta_*^n(r)&=\frac{1}{r^{n-2}}\int\limits_1^r\left[\tau^{n-3}\frac{R(\sigma(\tau))}{n-1}-\tau^{n-1}\frac{\overline{R}}{n-1}\right]_*\exp\left(-\int\limits_{\tau}^r\frac{s\left(\btr{\s'(s)}^{2}_{\s(s)}\right)^*}{4(n-1)}\,ds\right)\,d\tau,\\
\delta_n^*(r)&=\frac{1}{r^{n-2}}\int\limits_1^r\left[\tau^{n-3}\frac{R(\sigma(\tau))}{n-1}-\tau^{n-1}\frac{\overline{R}}{n-1}\right]^*\exp\left(-\int\limits_{\tau}^r\frac{s\left(\btr{\s'(s)}^{2}_{\s(s)}\right)_*}{4(n-1)}\,ds\right) \,d\tau,
\end{align*}
where for a function $f:M\to\R$ we define
\[
	f_*(r):=\inf\limits_{S^{n-1}}f(r,\cdot)\text{, and }f^*(r):=\sup\limits_{S^{n-1}}f(r,\cdot).
\]
As $\delta_*^n(t)$ remains finite for all radii $r\ge1$, a sufficient condition for \eqref{eq_hamiltionianPDE_appendix} to be parabolic is that the lower bound remains strictly positive for all times $t\ge t_0=1$. This depends on our choice of initial data $\varphi\equiv N(1,\cdot)$. More precisely, similar as in Lin \cite{Lin}, \eqref{eq_hamiltionianPDE_appendix} is parabolic if 
\[
	0<\varphi<\frac{1}{\sqrt{K}},
\]
where 
\[
	K:=\sup\limits_{1\le r< \infty}\left\{ -\int\limits_1^r\left[\tau^{n-3}\frac{R(\sigma(\tau))}{n-1}-\tau^{n-1}\frac{\overline{R}}{n-1}\right]_*\exp\left(\int\limits^{\tau}_1\frac{s\left(\btr{\s'(s)}^{2}_{\s(s)}\right)^*}{4(n-1)}\,ds\right)\,d\tau\right\}.
\]
Note that $K$ is finite by \eqref{eq_integrabilityR} and the exponential convergence of $(\s(r))$ to the round metric. Proceeding from here exactly as in Lin \cite{Lin} Schauder theory yields global existence, and the desired decay on $N$ follows from Schauder estimates. Note that the barriers and $K$ are consistent with the work of Lin \cite{Lin} for $n=3$.

Let us now also briefly consider the asymptotically hyperbolic case with $\Lambda=-\frac{(n-1)(n-2)}{2}$: Note that the equations regarding $N$ and $\omega$ as above remain unchanged, as the only difference in both cases is the difference in the prescribed scalar curvature $\overline{R}$, which differs by the constant $2\Lambda$. Now, to compensate for $\Lambda<0$, we proceed as Jang in \cite{Jang} and consider $u$ such that
\[
	N=\frac{u}{\sqrt{1+r^2}},
\]
and define $\widetilde{\omega}=u^{-2}$. Recall that $u$ is chosen precisely as in our definition of asymptotic hyperbolicity, cf. Section \ref{sec_prelim}, and we note that $\omega=(1+r^2)\widetilde{\omega}$. Thus, \eqref{eq_evolutionomega} yields
\begin{align}
\begin{split}\label{eq_evolutionomega2}
\partial_r\widetilde{\omega}=&\,\frac{1}{(n-1)r(1+r^2)}\widetilde{\omega}^{-1}\Delta_{\sigma(r)}\widetilde{\omega}+\frac{3}{(n-1)}\frac{u}{r(1+r^2)}{}^{\s(r)}\!\nabla u\cdot{}^{\s(r)}\!\nabla \widetilde{\omega}\\&\,-\left(\frac{n-2}{r}+\frac{2r}{1+r^2}+\frac{r}{4(n-1)}\btr{\s'(r)}^2_{\s(r)}\right)\widetilde{\omega}
+\frac{1}{r(1+r^2)}\left(\frac{R(\s(r))-\overline{R}r^2}{(n-1)}\right).
\end{split}
\end{align}
ODE comparrison yields for all $r_0\ge 1$
\begin{align*}
u^{-2}(r,\cdot)&\ge \delta_*^n(r)+\left(\frac{r_0}{r}\right)^{n-2}\frac{1+r_0^2}{1+r^2}\left(u^*(r_0)^{-2}-\delta_*^n(r_0)\right)\exp\left(-\int\limits_{r_0}^r\frac{s\left(\btr{\s'(s)}^{2}_{\s(s)}\right)^*}{4(n-1)}\, ds\right),\\
u^{-2}(r,\cdot)&\le \delta_n^*(r)+\left(\frac{r_0}{r}\right)^{n-2}\frac{1+r_0^2}{1+r^2}\left(u_*(r_0)^{-2}-\delta_n^*(r_0)\right)\exp\left(-\int\limits_{r_0}^r\frac{s\left(\btr{\s'(s)}^{2}_{\s(s)}\right)_*}{4(n-1)}\, ds\right),
\end{align*}
with barriers
\begin{align*}
\delta_*^n(r)&=\frac{1}{r^{n-2}(1+r^2)}\int\limits_1^r\left[\tau^{n-3}\frac{R(\sigma(\tau))}{n-1}-\tau^{n-1}\frac{\overline{R}}{n-1}\right]_*\exp\left(-\int\limits_{\tau}^r\frac{s\left(\btr{\s'(s)}^{2}_{\s(s)}\right)^*}{4(n-1)}\,ds\right)\,d\tau,\\
\delta_n^*(r)&=\frac{1}{r^{n-2}(1+r^2)}\int\limits_1^r\left[\tau^{n-3}\frac{R(\sigma(\tau))}{n-1}-\tau^{n-1}\frac{\overline{R}}{n-1}\right]^*\exp\left(-\int\limits_{\tau}^r\frac{s\left(\btr{\s'(s)}^{2}_{\s(s)}\right)_*}{4(n-1)}\,ds\right)\,d\tau.
\end{align*}
As in the work of Jang \cite{Jang}, parabolicity for the evolution of $u$ is ensured for initial data $\varphi$ at $r=1$, if
\[
	0<\varphi<\frac{1}{\sqrt{K}},
\]
where
\[
K:=\sup\limits_{1\le r< \infty}\left\{ -\int\limits_1^r\left[\tau^{n-3}\frac{R(\sigma(\tau))}{n-1}-\tau^{n-1}\frac{\overline{R}}{n-1}\right]_*\exp\left(\int\limits^{\tau}_1\frac{s\left(\btr{\s'(s)}^{2}_{\s(s)}\right)^*}{4(n-1)}\,ds\right)\,d\tau\right\}.
\]
Note that $K$ is the same constant as in the asymptotically flat case (up to the difference in $\overline{R}$ by $2\Lambda$, but as $\Lambda<0$ this yields a negative contribution only decreasing $K$) and hence finite by our assumption. To prove global existence and establish the desired decay on $u$, one then proceeds as outlined by Jang in \cite{Jang}. Note that the barriers and $K$ are consistent with the work of Jang \cite{Jang} for $n=3$.
\bibliography{refs-IDS}

\nopagebreak
\bibliographystyle{plain}
\end{document}